\documentclass[11pt,reqno]{article}
\usepackage[top=1.8cm, bottom=1.8cm, left=2cm, right=2cm]{geometry}

\usepackage{float}
\usepackage{graphicx}
\usepackage{subcaption}
\usepackage{pgfplots}
\pgfplotsset{compat=1.18}
\usepackage{caption}
\usepackage{amsmath}
\numberwithin{equation}{section}
\usepackage{amssymb}
\usepackage{amsfonts}
\usepackage{dsfont}
\usepackage{cases}
\usepackage[hyphens]{url}
\usepackage{amsthm}
\usepackage{upgreek}
\usepackage{xcolor}
\usepackage{bm}

\usepackage{mathtools}

\usepackage{hyperref}
\hypersetup{
    colorlinks=true,
    citecolor=blue,
    linkcolor=blue,
    urlcolor=orange
}

\usepackage{cleveref}

\usepackage[
    backend=biber,
    style=authoryear,
    giveninits=true,
    uniquename=true,
     dashed=false,
      maxnames=99,
    minnames=99
]{biblatex}

\DeclareMathOperator*{\argmax}{arg\, max}
\DeclareMathOperator*{\argmin}{arg\, min}

\newtheorem{theorem}{Theorem}[section]
\newtheorem*{theorem*}{Theorem}

\newtheorem{lemma}{Lemma}[section]

\newtheorem{proposition}{Proposition}[section]
\newtheorem{assumption}{Assumption}[section]
\newtheorem{definition}{Definition}[section]
\newtheorem{remark}{Remark}[section]

\newcommand{\Eb}{\mathbb{E}}

\newcommand{\Rb}{\mathbb{R}}

\newcommand{\Qc}{\mathcal{Q}}
\newcommand{\Ac}{\mathcal{A}}

\newcommand{\Fc}{\mathcal{F}}
\newcommand{\Dc}{\mathcal{D}}

\newcommand{\Bc}{\mathcal{B}}

\newcommand{\Ph}{\widehat{\Pi}}
\newcommand{\ph}{\widehat{\pi}}
\newcommand{\Xih}{\widehat{\Xi}}
\newcommand{\Xh}{\widehat{X}}
\newcommand{\xh}{\widehat{\xi}}

\newcommand{\dd}{\mathrm{d}}
\newcommand{\ee}{\mathrm{e}}

\begin{document}
\title{Equilibrium for regular-singular control under mean-variance criterion: A unified approach via control laws}

\author{Zongxia Liang\thanks{Department of Mathematical Sciences, and Center for Insurance and Risk Management, School of Economics and Management, Tsinghua University,  China. Email: liangzongxia@tsinghua.edu.cn} 
	\hspace{2ex}
	Xiaodong Luo\thanks{Department of Applied Mathematics, The Hong Kong Polytechnic University, Kowloon, Hong Kong. Email: xiaodong.luo@polyu.edu.hk}
    \hspace{2ex}
	Jiayu Zhang\thanks{Department of Mathematical Sciences,  Tsinghua University,  China. Email: zjy23@mails.tsinghua.edu.cn}}

\date{\today}

\maketitle

\begin{abstract}
This paper studies a class of mixed regular-singular control problems under mean-variance criteria, where the drift and diffusion are allowed to depend on both the regular control and the level of singular control. We seek time-consistent equilibrium strategies in an intrapersonal game setting and propose a novel equilibrium notion  where both regular and singular controls are unified via control laws, or equivalently, mappings on the augmented state space. Under which, we derive a verification theorem and necessary conditions providing a full mathematical characterization of the equilibrium.  We apply the theory to a reinsurance problem. The equilibrium solution turns out to be nontrivially coupled, where the regular control depends on the level of singular control, and the free boundary of the singular control switches dynamically in accordance with the variation of the regular control expression.  The free boundary is characterized in a piecewise semi-explicit manner and is proved to be $C^{1}$ across the switching point. In the degenerate case $\alpha_2=0$, the coupled solution reduces to the combination of two independent single-control equilibria and coincides with the limit as the parameter tends to zero.
\end{abstract}


\medskip

\section{Introduction}
Bellman's principle of optimality, the very cornerstone of stochastic control, is violated when the objective functional depends on future preferences in a way that destroys the recursive property fundamental to dynamic programming. Consequently, a strategy regarded as optimal today may fail to be optimal when re-evaluated at a future time. This gives rise to the issue of time inconsistency, which has become a central topic in modern stochastic control theory. Various sources of time inconsistency have been identified in the literature, including mean-variance optimization, non-exponential discounting and probability distortion, see, e.g.,  \textcite{strotz1955myopia}, \textcite{basak2010dynamic} and \textcite{barberis2012model}, and references therein. Among these formulations, the mean-variance (MV) criterion occupies a particularly important position due to its central role in finance, actuarial science and risk management.

The mean-variance (MV) criterion, initially proposed by \textcite{Markowitz1952} within the single-period portfolio selection framework, stands as one of the most influential criteria in the field of quantitative finance. When this criterion is extended to multi-period or continuous-time settings, the variance term inherent in the MV criterion naturally gives rise to the issue of time inconsistency. One line of research addresses this difficulty by considering pre-commitment strategies , whereby the decision maker optimizes the objective functional at the initial time and commits to implementing the resulting policy throughout the entire planning horizon. The corresponding optimization problems can be solved via martingale methods (see, e.g.,  \textcite{BajeuxBesnainou1998}, \textcite{gao2026dynamic} and etc) or embedding techniques (see, e.g.,  \textcite{zhou2000continuous}, \textcite{li2000optimal}, and etc). However, because the decision maker's preferences and risk attitudes may evolve over time, requiring a rational agent to adhere to a policy that is no longer optimal from the perspective of a future self lacks behavioral consistency.

A different perspective can be traced back to Strotz's intrapersonal game formulation, in which the successive incarnations of the same decision maker are viewed as noncooperative players in a dynamic game (see \textcite{strotz1955myopia}). This perspective later became the game-theoretic foundation for time-inconsistent stochastic control, replacing global optimality with the notion of a time-consistent equilibrium. \textcite{ekeland2006being} first introduced a rigorous definition of intra-personal equilibrium in continuous time. Building on this game-theoretic perspective, \textcite{basak2010dynamic} first systematically derived time-consistent equilibrium strategies for the continuous-time mean--variance portfolio selection problem, providing an analytical characterization of the equilibrium portfolios in general incomplete markets. \textcite{bjork2010general} further developed the game-theoretic framework by establishing an extended Hamilton--Jacobi--Bellman (HJB) system together with a corresponding verification theorem for a broad class of Markovian time-inconsistent stochastic control problems. The equilibrium framework was subsequently generalized to incorporate state-dependent risk aversion by \textcite{bjork2014mean}. Meanwhile, \textcite{hu2012time,hu2017time} developed a rigorous equilibrium theory for time-inconsistent stochastic linear--quadratic control problems, including the characterization and uniqueness of equilibrium. Since then, the equilibrium framework has been widely applied to a broad range of problems, including portfolio selection (see, e.g., \textcite{kronborg2015inconsistent}, \textcite{chen2019time},  \textcite{wang2025time} and etc ), optimal dividend problem (see, e.g., \textcite{zhao2014dividend}, \textcite{li2016equilibrium}, \textcite{cao2025equilibrium} and etc), asset--liability management (see, e.g., \textcite{chiu2020mean} and etc), and reinsurance--investment (see, e.g., \textcite{lin2016time} and etc). 

It is worth noting that the aforementioned studies are all confined to regular control problems, while research on time‑inconsistent singular control problems has only gradually emerged in recent years. \textcite{liang2024equilibria} first established an equilibrium framework for time‑inconsistent singular control with non‑exponential discounting preferences, proposed the corresponding equilibrium concept, and proved the corresponding verification theorem. Subsequently, \textcite{bodnariu2025time} further introduced the notion of mild thresholds and extended the equilibrium strategy to more general classes of controls. In another stream of research, \textcite{dai2024dynamic} proposed a different definition of equilibrium in a time-inconsistent mean–variance portfolio selection problem with transaction costs and characterized the equilibrium trading strategies. Meanwhile, the theory of time-inconsistent singular control has also been applied to a variety of financial and insurance models.  \textcite{liang2025stackelberg} investigated the optimal reinsurance problem under the mean–variance criterion, while \textcite{cao2026equilibrium} analyzed the optimal dividend problem within the mean–variance framework, and they respectively provided the corresponding characterization of equilibrium strategies and verification results. More recently, \textcite{caoyue2026equilibrium} investigated an equilibrium singular dividend control problem arising from ambiguity aggregation of heterogeneous discount rates, and \textcite{dai2026time} studied equilibrium singular control problems involving a running minimum process.

Despite the recent progress in the theory of time-inconsistent singular control, singular interventions alone are insufficient to capture many practical decision-making problems. In a wide range of applications, including finance, insurance, and resource management, decision makers often employ both continuous adjustments and singular interventions, naturally leading to mixed regular--singular control problems (see, e.g., \textcite{taksar1998optimal}, \textcite{chen2021free}, \textcite{qiu2023optimal} and etc). However, previous literature is predominantly confined to time-consistent problems, and a general equilibrium framework for mixed regular--singular control problems under the time-inconsistent mean--variance criterion is still lacking. 

Motivated by the above observations, this paper investigates a class of  regular--singular stochastic control problems under the mean-variance criterion in which the drift and diffusion coefficients of the controlled state process are jointly affected by both regular and singular controls. Our main contributions are summarized as follows.

First, we develop a closed-loop equilibrium framework for this general class of
time-inconsistent regular--singular stochastic control problems based on the unified notion of control laws. In our framework, both the regular and singular controls are simultaneously and jointly generated by the corresponding control laws. Control laws are equivalently mappings on the augmented state space, where the regular control law maps into $\mathbb{R}$ while the singular control law is equivalently a mapping into the binary-valued set of ``waiting'' and ``intervening''. In contrast to existing literature on regular-singular control (see, e.g., \textcite{qiu2023optimal} and etc), where the regular control is usually specified by a feedback function while the singular control is given separately as a stochastic process, our framework provides a unified perspective on these two types of controls. Motivated by the fact that the two controls jointly influence both the drift and diffusion coefficients of the controlled state process and are therefore intrinsically coupled, our framework simultaneously perturbs the regular and singular controls. The approach is then different from that in \textcite{liang2023weak}, where the regular control and the stopping time are perturbed separately in a mixed control problem of regular control and stopping. The corresponding definition of equilibrium strategies in our framework is presented in Definition~\ref{equ_law}.

Second, under suitable regularity and integrability conditions, we establish both a verification theorem and a necessary condition for equilibrium strategies; see Theorems~\ref{verification thm} and \ref{necessary thm}. These results provide a rigorous theoretical foundation for the proposed equilibrium framework by characterizing equilibrium strategies from both sufficient and necessary perspectives.

Finally, we derive the explicit equilibrium solutions within the specific context of reinsurance. A distinctive feature of the resulting equilibrium is the mutually coupling between the regular and singular controls. On one hand, the equilibrium regular control explicitly depends on the singular control (see \eqref{pi_equilibrium}); on the other hand, the free boundary of the singular control is determined by the regular control. Consequently, a regime-switching phenomenon occurs in the singular free boundary, driven by the structural transitions in the regular control expression. Moreover, although the free boundary $k(t)$ is characterized in a piecewise semi-explicit manner according to different forms of the regular control, it still satisfies $k(t) \in C^1(\{t\in [0,T]:k(t)>0\})$, indicating that the switching of its expression does not destroy first-order smoothness (see Proposition~\ref{C1} and Remark~\ref{c1}). To the best of our knowledge, such a switching phenomenon has not been reported in the existing literature on time-consistent regular--singular stochastic control. In the available explicit solutions, the singular free boundary is either independent of the regular control or only implicitly coupled with it through shared parameters, rather than being determined by the active regular-control regime;  see, e.g., \textcite{guo2004constrained}, \textcite{hafayed2016mckean} and \textcite{guan2023dynamic}. Furthermore, we investigate the degenerate cases involving single controls and the vanishing coupling parameter ($\alpha_2 = 0$). The results indicate that as the coupling parameter vanishes, the equilibrium solution converges to the limit of the two single-control scenarios. In other words, when the coupling parameter disappears, the equilibrium solutions become fully decoupled and can be solved separately from the respective single-control problems.

The remainder of the paper is organized as follows. In Section~\ref{model}, we formulate the model and introduce the equilibrium framework. Section~\ref{verfication} provides a verification theorem, and Section~\ref{necessary} establishes a necessary condition for equilibrium solutions. Explicit equilibrium solutions and their structural properties,  and  several degenerate cases are studied in Section~\ref{Example}.

\section{Model}
\label{model}
\subsection{Problem setting}
We consider a filtered probability space 
$(\Omega, \mathcal{F}, \{\mathcal{F}_t\}_{t \ge 0}, \mathbb{P})$
satisfying the usual conditions, and supporting a standard Brownian motion 
$\{B_t\}_{t \ge 0}$. We introduce two distinct types of controls: a \textit{regular control} $\pi = \{\pi_t\}_{t \ge 0}$  and a \textit{cumulative singular control} $\xi = \{\xi_t\}_{t \ge 0}$. We consider the controlled state (risk exposure/wealth) process $X^{\pi,\xi}$ under the regular control $\{\pi_t\}_{t \ge 0}$ and singular control $\{\xi_t\}_{t\ge0}$ defined by the stochastic differential equation(abbr. SDE):
\begin{equation}\label{X}
\begin{cases}
\dd X_s^{\pi,\xi}
= \mu\big(X_s^{\pi,\xi}, s, \xi_s, \pi_s\big)\, \dd s
+ \sigma\big(X_s^{\pi,\xi}, s, \xi_s, \pi_s\big)\, \dd B_s
- \dd \xi_s, \quad s \in [0,T], \\
X_{0-}^{\pi,\xi} = x_0,
\end{cases}
\end{equation}
where the functions $\mu(\cdot,\cdot,\cdot,\cdot)$ and $\sigma(\cdot,\cdot,\cdot,\cdot)$ are given deterministic continuous functions. The {regular control}  $\pi = \{\pi_s\}_{s \ge 0}$ is a progressively measurable process taking values in $U_1$, and  $U_1 \subset \Rb$ is a compact set representing the admissible region for the regular control. The { singular control} $\xi = \{\xi_s\}_{s \ge 0}$ is a non-decreasing, c\`adl\`ag, $\{\Fc_t\}_{t\geq 0}$-adapted process with $\xi_{0-} = 0$ and bounded from above by a maximum capacity $m>0$. 

We consider a time-inconsistent mean-variance type objective functional.  For any admissible control pair $(\pi,\xi)$, the   objective functional is 
\begin{equation}\label{J0}
J(x,t,y;\pi,\xi)
:= \Eb_{x,t,y}\big[X_T^{\pi,\xi}\big]
+ \gamma \operatorname{Var}_{x,t,y}[X^{\pi,\xi}_{T}]+\theta \Eb_{x,t,y}\int_{t}^{T}c(X_{s-}^{\pi,\xi},s,\xi_{s-})e^{\rho(T-s)} \circ \dd\xi_{s},
\end{equation}
where 
{\small \begin{equation*}
\int_{t}^{T}c(X_{s-}^{\pi,\xi},s,\xi_{s-})e^{\rho(T-s)} \circ \dd\xi_{s}: =  \int_{t}^{T}c(X_{s-}^{\pi,\xi},s,\xi_{s-})e^{\rho(T-s)}  \dd\xi^c_{s} + \sum_{s\in[t,T]}\int_{0}^{\Delta \xi_s} c(X_{s-}^{\pi,\xi}-z,s,\xi_{s-}+z)\ee^{\rho(T-s)}\dd z.
\end{equation*}}
This paper is to minimize  the  objective functional $ J(x,t,y;\pi,\xi)$, i.e., solving the following  problem:
\begin{eqnarray}\label{OP1}
\inf_{\pi,\xi} \{J(x,t,y;\pi,\xi) \}.  
\end{eqnarray}
The discounted cumulative integral involving $c(\cdot,\cdot,\cdot)$ represents the total expected running cost incurred by executing the singular control policies over the remaining horizon, and $\xi^c$ denote the continuous part of $\xi$. This definition guarantees the additivity of intervention costs with respect to jump sizes and prevents the cost functional from depending on the particular representation of a singular intervention. Such path-independent specifications of jump costs are commonly adopted in the singular stochastic control literature (see, e.g., \textcite{zhu1992generalized}). 
The positive constants $\gamma$ (the risk aversion parameter) and $\theta$  serve as weighting factors that balance the components of the performance functional, thereby shaping the relative penalization of state variability, expected terminal outcomes and control activity. The parameter $\rho > 0$ is the discount rate, capturing the intertemporal valuation of future payoffs and ensuring the well-posedness of the infinite-horizon preference structure. The conditional expectation $\Eb_{x,t,y}$ and $\operatorname{Var}_{x,t,y}$ is conditioning on $X_{t-}^{\pi,\xi} = x$ and $ \xi_{t-} = y$.

\begin{remark}[Well-posedness of the terminal condition]
At the terminal time $t = T$, the objective functional simplifies to a static optimization over the terminal jump size $\Delta \xi_T$. Absent the global upper bound (i.e., $\xi_T \in [0, +\infty)$), the cost functional $J(x, T, y;\pi,\xi)$ can potentially be decreasing with respect to $y$, depending on the configurations of $ \theta$ and $c(x,t,y)$. Under such configurations, the minimization problem immediately suffers from an unbounded descent:
\begin{equation}
\inf_{\Delta \xi_T \ge 0} J(x, T, y;\pi,\xi) = -\infty,
\end{equation}
which implies that the terminal control expands to infinity, rendering the boundary condition of the HJB equations ill-defined. Consequently, the compact constraint $\xi_T \in [0, m]$ is indispensable to guarantee the existence of an equilibrium strategy and the closure of the verification theorem.

We note that if $J$ is strictly convex with a finite minimum under $m = +\infty$, our main analytical results remain perfectly valid. Nevertheless, to ensure uniform well-posedness against unbounded breakdowns under arbitrary parameter pair $( \gamma, \theta)$, this paper primarily adopts the bounded setting $m$. Practically, this setting aligns with realistic risk management constraints. The bound $m$ naturally represents the maximum capacity for purchasing reinsurance, or the hard budget constraint of a corporate risk reserve fund. Therefore, the compact constraint $\xi_T \in [0, m]$ is both mathematically indispensable and practically grounded.
\end{remark}

To ensure the analytical tractability and uniqueness of the optimal strategies for the generalized system \eqref{X} with \eqref{OP1}, we introduce the following structural configuration on the decision-making rule.

\begin{assumption}\label{asm:tie-breaking}
In the event that multiple singular control actions yield identical performance outcomes within the objective functional $J$, the controller consistently opts for the maximal admissible intervention capacity.
\end{assumption}

\begin{remark}[Technical Necessity of Assumption \ref{asm:tie-breaking}]
At the terminal horizon $t =T$, there may exist a degenerate region where the performance functional $J$ becomes insensitive to the singular control level. While any singular action within this domain yields identical performance, this non-uniqueness prevents the determination of a well-defined terminal condition for the extended HJB equations. 

Assumption \ref{asm:tie-breaking} operates as a systematic tie-breaking rule that selects the maximal admissible intervention capacity via a standard measurable selection argument. This device uniquely determines the terminal boundary condition, which is a prerequisite for deriving closed-form equilibrium strategies. Crucially, this boundary convention modifies neither the admissible control set nor the interior optimality conditions.
\end{remark}

\subsection{Admissible Control Law and Admissible Regular-Singular Control Pair}

Because the  regular–singular control problem \eqref{OP1} involves a mean-variance (MV) formulation of the terminal state $X_T^{\pi,\xi}$, it
is  a time-inconsistent stochastic control problem, we seek an equilibrium for the problem. We formulate the equilibrium framework of   Problem \eqref{OP1} in the sense of closed-loop control. To do this, we first design admissible control law and the admissible regular-singular control pair in this subsection, then we propose  the equilibrium framework in  next subsection. \ \ 
Define the state-time-(singular) control space of the framework by $$ \Qc := \Rb \times [0,T] \times [0,m].$$
\noindent
 The regular control process $\pi$ is generated by some regular control law $\Pi = \Pi(\cdot,\cdot,\cdot)$ as a feedback function of the triple state-time-(singular) control, while the singular control process is generated by some singular control law $\Xi$ as a division of the space $\Qc$. \\ 
 \vskip 3pt 
 First, the admissible regular-singular control pairs for Problem (\ref{OP1}) are defined by the following. The definition of admissible pairs of perturbations is obtained with an additional growth constraint, which will later be used in our equilibrium definition.
\begin{definition}[Admissible regular-singular control pair]
\label{admissible pair}
Fix initial $(x,t,y) \in \Qc$, let $\{\pi_s\}_{s\in [t,T]}$ be a progressively measurable process and $\{\xi_s\}_{s\in [t,T]}$ be a non-decreasing, c\`adl\`ag, $\{\Fc_s\}_{s\ge t}$-adapted process. Then $(\{\pi_s\}_{s\in [t,T]}, \{\xi_s\}_{s\in [t,T]})$ is called the  admissible regular-singular control pairs  on $[t,T]$ if the followings hold:

(1) $\xi_{t-}=y$ and the SDE \begin{equation}
\label{SDE}
\begin{cases}
     \dd X_s^{\pi,\xi} = \mu(X_s^{\pi,\xi},s,\xi_s,\pi_s) \dd s+\sigma(X_s^{\pi,\xi},s,\xi_s,\pi_s) \dd B_s - \dd\xi_s,\,\forall s \in [t,+\infty),\\
     X_{t-}^{\pi,\xi} = x
\end{cases}\end{equation}
has a unique strong solution $\{X_s^{\pi,\xi}\}_{s\in[t,T]}$.

(2) The strong solution $X^{\pi,\xi}$ given in  \eqref{SDE} satisfies 
\begin{equation}
    \begin{aligned}
        \Eb_{x,t,y}(X_T^{\pi,\xi})^2 < + \infty,\quad  \Eb_{x,t,y} \left[\int_t^T \ee^{\rho(T-s)}c(X^{\pi,\xi}_{s-},s,\xi_{s-}) \circ \dd \xi_s \right]< +\infty.
    \end{aligned}
\end{equation}

(3) For any $s \in [t,T]$, $\Delta \xi_s$ is $\Fc_{s-}$ measurable.

(4) $\pi_s \in U_1, \ \ \xi _s \in [0,m]$ almost surely for any $s\in [t,T].$

Denote the set of all admissible regular-singular control pairs on $[t,T]$ by $\Dc_{[t,T]}$. If there exist constants $h \in (0,T-t]$ and $M>0$ such that $\xi$ additionally satisfies
\begin{equation}
    \label{pert_xi}
    \xi_{(t+h')-}-\xi_t \le Mh', a.s., \forall h' \in (0,h),
\end{equation}
then we call $(\{\pi_s\}_{s\in[t,t+h)}, \{\xi_s\}_{s\in[t,t+h)})$ an admissible  regular-singular control pair of perturbations on $[t,t+h)$. We denote the set of all admissible  regular-singular control pairs  of perturbations  on $[t,t+h)$ by $\Dc'_{[t,t+h)}$.
\end{definition}
\begin{remark}
    The control pair defined in  Definition \ref{admissible pair} constitutes an open‑loop control, as it is specified directly as adapted processes without any dependence on state feedback. Growth condition similar to \eqref{pert_xi} is standard in the literature on time-inconsistent singular controls, see \textcite{liang2024equilibria} and \textcite{dai2026time}. 
\end{remark}

 Second, for the closed-loop formulation, we introduce the definition of regular-singular control laws. In the closed-loop formulation for time-inconsistent regular control problems, the feedback control $\Pi: \Qc \rightarrow \Rb$ is a mapping from $\Qc$ into $\Rb$, and is termed as the regular control law in  \textcite{bjork2017time} and \textcite{bjork2021time}. While in time-inconsistent singular control literature \textcite{liang2024equilibria} and \textcite{liang2025stackelberg}, the partition $\Xi = (W^{\Xi}, P^{\Xi})$ of $\Qc$ is termed the singular control law. $\Xi$ is equivalently a mapping from $\Qc$ into the binary-set of ``waiting'' and ``intervening''. Hence, control laws can be unified as mappings on the augmented state space $\Qc$, which leads to the following definition for regular-singular control laws.
\begin{definition}[Admissible regular-singular control law]
    Let $\Pi = \Pi(\cdot,\cdot,\cdot)$ be a borel-measurable function on $\Qc$, and $\Xi = (W^{\Xi},P^{\Xi})$ be a division of $\Qc$ with $W^{\Xi}$ denoting the waiting region and $P^{\Xi} = (W^{\Xi})^c$ be the action region (the complement of the waiting region). The regular-singular control law  $(\Pi,\Xi)$ is admissible if the following conditions hold:
    
    (1) The Skorohod reflection problem 
\begin{equation}
   \!\!  \!\!   \!\!   \!\!  \begin{cases}
    \label{xi}
        \dd X_s^{\Pi,\Xi} = \mu(X_s^{\Pi,\Xi},s,\xi_s^{\Pi,\Xi},\Pi(X_s^{\Pi,\Xi},s,\xi_s^{\Pi,\Xi})) \dd s+\sigma(X_s^{\Pi,\Xi},s,\xi_s^{\Pi,\Xi},\Pi(X_s^{\Pi,\Xi},s,\xi_s^{\Pi,\Xi})) \dd B_s - \dd\xi_s^{\Pi,\Xi},\\ \hskip 0.5cm \,\forall s \in [t,T],\\
        (X_s^{\Pi,\Xi},s,\xi_s^{\Pi,\Xi}) \in \overline{W^{\Xi}},\,\, \forall s \in (t,T],\\
        X_{t-}^{\Pi,\Xi} = x,\\
        \xi_s^{\Pi,\Xi} = y +\int_t^s\mathbf{1}_{\{ (X_r^{\Pi,\Xi},r,\xi_r^{\Pi,\Xi}) \in P^{\Xi}\}} \dd \xi_r^{\Pi,\Xi},\,\, \forall s \in [t,T],
    \end{cases} 
    \end{equation}
has a unique strong solution $(X^{x,t,y;\Pi,\Xi}, \xi^{x,t,y;\Pi,\Xi}):=(\{X_s^{\Pi,\Xi}\}_{s \in [t,T]}, \{\xi_s^{\Pi,\Xi}\}_{s \in [t,T]})$. The processes $\pi^{x,t,y;\Pi,\Xi}:= \{\Pi(X_s^{\Pi,\Xi},s,\xi^{\Pi,\Xi}_s)\}_{s\in [t,T]}$ and $\xi^{x,t,y;\Pi,\Xi}$ are respectively called  the regular control and singular control process simultaneously generated by the pair of control law $(\Pi,\Xi)$ at the initial $(x,t,y) \in \Qc$.

(2) The strong solution $(X^{x,t,y;\Pi,\Xi}, \xi^{x,t,y;\Pi,\Xi})$ given in \eqref{xi} satisfies 
\begin{equation}
    \begin{aligned}
        \Eb_{x,t,y}(X_T^{\Pi,\Xi})^2 < + \infty,\quad  \Eb_{x,t,y} \left[\int_t^T \ee^{\rho(T-s)}c(X_{s-}^{\Pi,\Xi},s,\xi^{\Pi,\Xi}_{s-}) \circ  \dd \xi^{\Pi,\Xi}_s\right] < +\infty.
    \end{aligned}
\end{equation}

(3) For any $t \in [0,T]$, $\Delta \xi_t$ is $\Fc_{t-}$ measurable.

(4) $\Pi(x,t,y) \in U_1, \ \ \ \xi_s^{x,t,y;\Pi,\Xi} \in [0,m]$ almost surely for any $(x,t,y) \in \Qc$ and $s\in[t,T]$.
\end{definition}
\noindent
Then the  objective functional of an admissible regular-singular control law $(\Pi,\Xi)$ is 
\begin{equation}
    \label{J2}
J(x,t,y;\Pi,\Xi):=\mathbb{E}_{x,t,y}[X^{\Pi,\Xi}_{T}]+ \gamma \operatorname{Var}_{x,t,y}[X^{\Pi,\Xi}_{T}]+\theta \mathbb{E}_{x,t,y}\int_{t}^{T}c(X^{\Pi,\Xi}_{s-},s,\xi^{\Pi,\Xi}_{s-})e^{\rho(T-s)} \circ \dd\xi^{\Pi,\Xi}_{s}.
\end{equation}

We next provide sufficient conditions on the drift and diffusion coefficients of the controlled state process to ensure that an arbitrary regular-singular control satisfies the first two admissibility conditions of Definition~\ref{admissible pair}.
\begin{proposition}
\label{square int}
    If there exists a constant $K>0$ such that 
    \begin{equation}
        \begin{aligned}
           & |\mu(x_1,t,y,u)-\mu(x_2,t,y,u)|+ |\sigma(x_1,t,y,u)-\sigma(x_2,t,y,u)| \le K|x_1-x_2|,  \label{Lipschitz}\\
           & \forall x_1,x_2 \in \Rb,\  \forall(t,y,u) \in [0,T] \times [0,m] \times U_1, \\
            &|\mu(x,t,y,u)|^2+|\sigma(x,t,y,u)|^2 + |c(x,t,y)| \le K(1+x^2),\ \forall (x,t,y,u) \in \Rb \times [0,T] \times [0,m] \times U_1,
        \end{aligned}
    \end{equation}
    then for any $(x,t,y)\in \Qc$, any progressively measurable process $\{\pi_s\}_{s \in [t,T]}$, any non-decreasing , c\`adl\`ag, $\{\Fc_s\}_{s\in [t, T]}$-adapted process $\{\xi_s\}_{s\in[t,T]}$ such that $\xi_{t-}=y$ and $\xi_s \in [0,m],\ \pi_s \in U_1,a.s., \ \forall s\in [t,T]$, we have that there exists a unique strong solution for the SDE \eqref{SDE} satisfying 
    \begin{equation}
            \Eb_{x,t,y} \left[\sup_{s \in [t,T]}(X_s^{\pi,\xi})^2\right] < +\infty,\quad \Eb_{x,t,y} \left[\int_t^T \ee^{\rho(T-s)}c(X^{\pi,\xi}_{s-},s,\xi_{s-}) \circ \dd \xi_s \right]< +\infty.
    \end{equation}
\end{proposition}
\begin{proof}
    See Appendix~\ref{pro2.1}.
\end{proof}

\subsection{Equilibrium Framework: Equilibrium Law and Equilibrium Control Pair}
In this subsection, we propose the equilibrium law and  equilibrium control pair for time-inconsistency control problem \eqref{OP1} in the weak equilibrium sense as follows.
\begin{definition}[Equilibrium law and equilibrium strategy]
\label{equ_law}
    An admissible regular-singular control law $(\Ph,\Xih)$ is called an equilibrium regular-singular control law (abbr.  equilibrium law ) for Problem \eqref{OP1} if the followings (1) and (2) hold. 
    
    (1) For any initial $(x,t,y) \in \Qc$ with $t<T$, there exists an $h_0 \in (0,T-t]$ such that for any admissible regular-singular control pair of perturbations $(u,\eta)\in \Dc'_{[t,t+h)}$ with $h < h_0, \eta_{t-} = y$, it holds 
    \begin{equation}
        \label{equ_J}
        \liminf_{h \downarrow 0} \frac{J(x,t,y;\pi^h,\xi^h) - J(x,t,y;\ph,\xh)}{h} \ge 0,
    \end{equation}
    where the perturbed regular-singular control pair $(\pi^h,\xi^h), \forall h \in (0,h_0)$ is given by
    \begin{equation}
    \begin{aligned}
    \label{pi^h,xi^h}
        \pi_s^h =
    \begin{cases}
         u_s, & \forall s \in [t,t+h),\\
         \pi_s^{X^{u,\eta}_{(t+h)-},t+h,\eta_{(t+h)-};\Ph,\Xih}, & \forall s \in [t+h,T],
    \end{cases} \\
    \xi^h_s =
    \begin{cases}
         \eta_s, & \forall s \in [t,t+h),\\
         \xi_s^{X^{u,\eta}_{(t+h)-},t+h,\eta_{(t+h)-};\Ph,\Xih}, & \forall s \in [t+h,T],
    \end{cases}
    \end{aligned}
     \end{equation}
    and the equilibrium regular-singular control pair $(\ph,\xh)$ is given by $\ph = \pi^{x,t,y;\Ph,\Xih}$ and $\xh:=\xi^{x,t,y;\Ph,\Xih}$.

    (2) For any $(x,t,y) \in \Qc$ with $t=T$, any $(u,\eta) \in \Dc_{[T,T]}$ with $\eta_{T-} = y$, we have $$
    J(x,t,y;u,\eta) - J(x,t,y;\ph,\xh) \ge 0. $$

    We call $J(x,t,y;\Ph,\Xih)=J(x,t,y;\ph,\xh)$ the corresponding equilibrium value function of $(\Ph,\Xih)$.
\end{definition}

\begin{remark}
    The control processes $\pi^{x,t,y;\Pi,\Xi}$ and $\xi^{x,t,y;\Pi,\Xi}$ are generated simultaneously, and the dependence of the state dynamics on the pair $(\Pi,\Xi)$ is highly  coupled. Consequently,  Definition \ref{equ_law} is structurally distinct from that in \textcite{liang2023weak}, where the equilibrium conditions are formulated separately for stopping controls and regular controls.
\end{remark}

\begin{remark}
    In the closed-loop formulation for time-inconsistent stochastic control, the mapping $\Pi: \Qc \rightarrow \Rb$ is termed as the regular control law (see \textcite{bjork2017time} and \textcite{bjork2021time}), while the partition $\Xi = (W^{\Xi}, P^{\Xi})$ of $\Qc$ is termed the singular control law (see \textcite{liang2024equilibria} and \textcite{liang2025stackelberg}).
\end{remark}
\begin{remark}
There is another approach to defining equilibrium for singular control problems in the literature, as adopted in \textcite{dai2024dynamic} and \textcite{cao2026equilibrium}. Recently, \textcite{liang2026timeconsistent} extended this approach in an annuitization and asset allocation problem where both regular and singular control are involved. However, their definition is formulated directly on the singular control process itself, which may not be appropriate for problems where the singular control variable also appears as a state variable in the dynamics or the objective, i.e., where not only $d\xi_t$ but also $\xi_t$ matters. This is precisely the case in this paper, as well as in  \textcite{liang2024equilibria} and  \textcite{liang2026timeconsistent}. In such settings, applying the direct definition leads to the issue that $X^{\pi^{h},\xi^{h}}_{r}\neq X^{\hat{\pi},\hat{\xi}}_{r}$ for $r>t+h$; consequently, the required relation  $J(X_{t+h},t+h,\xi_{(t+h)-};\pi^{h},\xi^{h})=J(X_{t+h},t+h,\xi_{(t+h)-};\ph,\xh)$ in the verification proof may fail. For problems of this type, the singular control law approach provides a more rigorous foundation for characterizing equilibrium.
\end{remark}

\section{Sufficient Conditions}
\label{verfication}
The objective of this section is to establish a verification theorem for the equilibrium of Problem \eqref{OP1}. The verification theorem provides a set of sufficient conditions under which a candidate value function and the associated control law constitute an equilibrium. This result serves as the main analytical tool for constructing explicit equilibrium solutions in the subsequent sections.

\begin{theorem}[Verification Theorem]
\label{verification thm}
Define the infinitesimal operators
\begin{equation}
    \begin{cases}
        (\Ac^{\pi} \varphi)(x,t,y) &:= \varphi_t(x,t,y)+\mu(x,t,y,\pi) \varphi_x(x,t,y)+\frac{1}{2}\sigma^2(x,t,y,\pi)\varphi_{xx}(x,t,y),\,\, \forall \varphi: \Qc \rightarrow \Rb, \\
        (\Ac^{\Pi} \varphi)(x,t,y) &:= (\Ac^{\Pi(x,t,y)} \varphi)(x,t,y), \,\, \forall \varphi:\Qc \rightarrow \Rb. 
    \end{cases}
\end{equation}
Given two functions $V$ and $g$, define a regular-control law $\Ph$ by 
\begin{equation}
\label{Ph_def}
    \begin{aligned}
        \Ph(x,t,y) = \argmin_{\pi\in U_1} \{(\Ac^{\pi} V)(x,t,y) + \sigma^2(x,t,y,\pi)g_x^2(x,t,y)\},\,\, \forall (x,t,y) \in \Qc,
    \end{aligned}
\end{equation}
and a singular control law $\Xih=(W^{\Xih} , P^{\Xih}     )$ by 
    \begin{align}
        W^{\Xih} & := \{(x,t,y) \in \Qc: \theta \ee^{\rho(T-t)}c(x,t,y) - V_x(x,t,y)+V_y(x,t,y)>0 \text{ or } y = m\}, \label{W}\\ 
        P^{\Xih} & := \{(x,t,y) \in \Qc: \theta \ee^{\rho(T-t)}c(x,t,y) - V_x(x,t,y)+V_y(x,t,y)=0, y < m \}.\label{P} 
    \end{align}
Assume the following conditions hold:
\vspace{2mm}

(1) $V,g \in C^{2,1,1}(\{(x,t,y) \in \Qc:t<T\})\cap C(\Qc)$.
\vspace{2mm}

(2) For all $(x,t,y) \in \Qc$, $V$ and $g$ solve the extended HJB systerms:
  \! \!\!\!  \begin{align}
&\min \bigl\{ \min_{\pi\in U_1} \{(\Ac^{\pi} V)(x,t,y)\!\! + \! \gamma \sigma^2(x,t,y,\pi) g_x^2(x,t,y)\},
        \theta \ee^{\rho(T-t)}c(x,t,y) \!-\! V_x(x,t,y) \!+\! V_y(x,t,y)\bigl\}  = 0,  \label{min}\\
& V(x,T,y) = x-    \max_{r \in [0,m-y]} \left (r-\theta \int_0^{r}c(x-z,T,y+z) \dd z \right),  \label{VT}\\
& (\Ac^{\Ph}g)(x,t,y) = 0, \, \forall (x,t,y) \in \overline{W^{\Xih}}, \, t<T,  \label{gW}\\
 &       g_x(x,t,y) = g_y(x,t,y),\, \forall (x,t,y) \in P^{\Xih}, \, t<T, \label{gP}\\
    &    g(x,T,y) = x- \sup \left(\argmax_{r \in [0,m-y]} \left (r-\theta \int_0^{r}c(x-z,T,y+z) \dd z \right)\right),\label{gT}
    \end{align}
   where $\sup A$  denotes the supremum of the set $A$.
    
(3) $(\Ph,\Xih)$ is an admissible control law.
\vspace{2mm}

(4) For any $(x,t,y) \in \Qc$, there exists an $h_0>0$, such that for any admissible regular-singular control pair of perturbations $(u,\eta) \in \Dc'_{[t,t+h)}, \forall h \in (0,h_0)$, the corresponding  perturbed regular-singular control pair, defined in \eqref{pi^h,xi^h}, still denote by $(u,\eta)$, for functions $\varphi = V,\ g \ g^2, \ \phi = \Ac^ug, \ g_y-g_x,\  g$, it holds that 
\begin{equation}
    \label{finite}
    \Eb_{x,t,y}[\sup_{s\in [t,t+h)}|\phi(X^{u,\eta}_s,s,\eta_s)|] + \Eb_{x,t,y}\left[ \int_t^{T} \varphi_x^2(X^{u,\eta}_s,s,\eta_s) \dd s\right] <+\infty.
\end{equation}
Then $(\Ph,\Xih)$ is an equilibrium regular-singular control law and $V$ is the corresponding equilibrium value function. Moreover, $g$ has the probabilistic interpretation 
\begin{equation}
    g(x,t,y) = \Eb_{x,t,y}\left[X^{x,t,y;\Ph,\Xih}\right].
\end{equation}
\end{theorem}

\begin{proof}
Fix any $(x,t,y) \in \Qc$, for simplicity, we denote $\Xh := X^{x,t,y;\Ph,\Xih}$, and $(\ph, \xh)$ is the  admissible control pair corresponding to $(\Ph,\Xih)$.
This proof consists of three steps:

\textbf{Step 1.}{ \bf  We show that $g$ has the probabilistic interpretation}.\ \

Applying It\^o's formula to $g$ yields
 {\small    \begin{equation}
    \label{1}
        \begin{aligned}
             g(\Xh_{T-},T,\xh_{T-}) & =  g(x,t,y) + \int_t^T (\Ac^{\Ph}g)(\Xh_{s-},s,\xh_{s-}) \dd s \\
            &+ \int_t^T \sigma(\Xh_{s-},s,\xh_{s-},\ph(\Xh_{s-},s,\xh_{s-})) g_x(\Xh_{s-},s,\xh_{s-}) \dd B_s + \int_t^T\big[g_y(\Xh_{s-},s,\xh_{s-})-g_x(\Xh_{s-},s,\xh_{s-})\big]\dd \xh_s^c \\
            &+ \sum_{s\in[t,T)}\int_0^{\Delta\xh_s} \big[g_y(\Xh_{s-}-z,s,\xh_{s-}+z)-g_x(\Xh_{s-}-z,s,\xh_{s-}+z)\big] \dd z.
        \end{aligned}
    \end{equation}}
    Using Equation \eqref{gW} and  $(\Xh_{s},s,\xh_{s}) \in \overline{W^{\Xih}},a.s., \forall s \in (t,T]$, we have that the first integral in \eqref{1} equals $0$. Using Equation \eqref{gP} and the fact that $(\Xh_{s-}-z,s,\xh_{s-}+z) \in P^{\Xih},a.s.,\  \forall z  \ 
    \in [0,\Delta\xh_s], s \in [t,T)$ if $\Delta\xh_s>0$
    , the last two integrals on the right-hand side of \eqref{1} are equal to zero. Note that the second integral in \eqref{1} is a martingale under Condition (4), taking expectation on both sides of \eqref{1}, we obtain
    $$  \Eb_{x,t,y}[g(\Xh_{T-},T,\xh_{T-})] = g(x,t,y).$$
    Using \eqref{gT} yields
    \begin{equation}
    \label{g}
        g(x,t,y) = \Eb_{x,t,y}\left[\Xh_{T-}- \sup \left (\argmax_{r \in [0,m-\xh_{T-}]} \left(r-\theta \int_0^{r}c(\Xh_{T-}-z,T,\xh_{T-}+z) \dd z \right) \right)\right] = \Eb_{x,t,y}[\Xh_T].
    \end{equation}

    \textbf{Step 2.} { \bf We show that 
    \begin{equation}
    \label{VJ}
        V(x,t,y) = J(x,t,y;\ph,\xh).
    \end{equation}}

We first analyze $\operatorname{Var}_{x,t,y}(\Xh_T)$. 
\begin{equation}
\label{var}
    \begin{aligned}
       \operatorname{Var}_{x,t,y}(\Xh_T) & = \Eb_{x,t,y}(\Xh_T^2) - (\Eb_{x,t,y}\Xh_T)^2 = \Eb_{x,t,y}[g^2(\Xh_T,T,\xh_T)] - g^2(x,t,y) \\
       & = \Eb_{x,t,y} \biggl\{ \int_t^T (\Ac^{\Ph} g^2)(\Xh_{s-},s,\xh_{s-}) \dd s + \int_t^T \big[(2g g_y)(\Xh_{s-},s,\xh_{s-})-(2gg_x)(\Xh_{s-},s,\xh_{s-})\big] \dd \xh^c_{s}  \\
       &\quad \quad  + \sum_{s\in[t,T]} \int_0^{\Delta\xh_s} \big[(2g g_y)(\Xh_{s-}-z,s,\xh_{s-}+z)-(2g g_x)(\Xh_{s-}-z,s,\xh_{s-}+z)\big]\dd z \biggl\} \\
       & = \Eb_{x,t,y} \left[ \int_t^T (\Ac^{\Ph}g^2)(\Xh_{s-},s,\xh_{s-}) \dd s\right],
    \end{aligned}
\end{equation}
where the first equality follows from \eqref{g}, the second equality follows from applying It\^o's formula to $g^2$ on $[t,T]$ and taking expectations, and the last equality follows from  the fact that $(\Xh_{s-}-z,s,\xh_{s-}+z) \in P^{\Xih},a.s., \forall z 
    \in [0,\Delta\xh_s], s \in [t,T)$ if $\Delta\xh_s>0$ , Equation \eqref{gP} and Condition (4).

Observe that $$(\Ac^{\Ph}g^2)(x,t,y) = (2g\Ac^{\Ph}g)(x,t,y) + \sigma^2(x,t,y,\ph(x,t,y)) g_x^2(x,t,y),$$
using Equations \eqref{gW}, \eqref{var},  and $(\Xh_{s},s,\xh_{s}) \in \overline{W^{\Xih}},a.s.$, we have
\begin{equation}
    \label{var2}
\operatorname{Var}_{x,t,y}(\Xh_T)=\Eb_{x,t,y} \left[ \int_t^T \sigma^2(\Xh_{s-},s,\xh_{s-},\ph(\Xh_{s-},s,\xh_{s-})) g_x^2(\Xh_{s-},s,\xh_{s-}) \dd s\right].
\end{equation}
Similar to \textbf{Step 1}, applying It\^o's formular to $V$ and taking expectations, we have 
\begin{equation}
    \begin{aligned}
         \Eb_{x,t,y}[V(\Xh_T,T,\xh_T)] & =  V(x,t,y) + \Eb_{x,t,y} \Biggl[\int_t^T(\Ac^{\Ph}V)(\Xh_{s-},s,\xh_{s-}) \dd s \\ 
        &  + \int_t^T\big[V_y(\Xh_{s-},s,\xh_{s-})-V_x(\Xh_{s-},s,\xh_{s-})\big]\dd \xh_s^c \\
        &+ \sum_{s\in[t,T]}\int_0^{\Delta\xh_s} \big[V_y(\Xh_{s-}-z,s,\xh_{s-}+z)-V_x(\Xh_{s-}-z,s,\xh_{s-}+z)\big] \dd z\Biggl]
        \end{aligned}
        \end{equation}
        \begin{equation}
        \label{2}
        \begin{aligned}
        & = V(x,t,y) + \Eb_{x,t,y} \Biggl[- \int_t^T \theta \ee^{\rho(T-s)}c(\Xh_{s-},s,\xh_{s-})\dd \xh_s^c - \sum_{s\in[t,T]}\int_0^{\Delta \xh_s}\theta \ee^{\rho(T-s)}c(\Xh_{s-}-z,s,\xh_{s-}+z) \dd z\\
        & - \gamma \int_t^T \sigma^2(\Xh_{s-},s,\xh_{s-},\ph(\Xh_{s-},s,\xh_{s-})) g_x^2(\Xh_{s-},s,\xh_{s-}) \dd s   \Biggl]\\
        & = V(x,t,y) - \gamma \operatorname{Var}_{x,t,y}(\Xh_T) - \Eb_{x,t,y} \Biggl[\int_t^T \theta \ee^{\rho(T-s)}c(\Xh_{s-},s,\xh_{s-}) \circ \dd \xh_s \Biggl],
    \end{aligned}
\end{equation}
where the second equality follows from Condition (4), Equations \eqref{min} and the definition of $W^{\Xih}$ as in \eqref{W}. The last equality follows from \eqref{var2}. 

Using Equalities \eqref{VT} and \eqref{2}, we have 
\begin{align}
\label{3a}
    V(x,t,y) = &  \Eb_{x,t,y}(\Xh_T)+ \gamma \operatorname{Var}_{x,t,y}(\Xh_T) +\Eb_{x,t,y} \Biggl[\int_t^T \theta \ee^{\rho(T-s)}c(\Xh_{s-},s,\xh_{s-}) \circ \dd \xh_s \Biggl]\nonumber \\
    &- \max_{r \in [0,m-\xh_T]} \left (r-\theta \int_0^{r}c(\Xh_T-z,T,\xh_T+z) \dd z \right).
\end{align}
Note that if the singular control is activated at time $T$, it is immediately driven to its upper bound of $\argmax\limits_{r \in [0,m-\Xh_{T-}]} \left (r-\theta \int_0^{r}c(\Xh_{T-}-z,T,\Xh_{T-}+z) \dd z \right)$, that is, $$\argmax_{r \in [0,m-\Xh_{T}]} \left (r-\theta \int_0^{r}c(\Xh_T-z,T,\xh_{T}+z) \dd z \right)=0.$$ Consequently, the last term of \eqref{3a} equals 0, thereby yielding the desired result.
\vskip 3pt
\textbf{Step 3.} {\bf We prove that $(\Ph,\Xih)$ is an equilibrium singular-regular control law}.

Fix any $(x,t,y) \in \Qc$ and any admissible singular-regular pair of perturbation $(u,\eta) \in \Dc'_{[t,t+h)}, h < h_0$, where $h_0$ is defined in Condition (4). We first analyze $J(x,t,y;u,\eta)$ as follows.

Using the conditional expectation formula, we have 
\begin{eqnarray}
     J(x,t,y;u,\eta)  &=& \Eb_{x,t,y}[g(X^{u,\eta}_{(t+h)-},t+h,\eta_{(t+h)-})] + \gamma
 \operatorname{Var}_{x,t,y}(X^{u,\eta}_T) \nonumber\\
    &&+ \theta \Eb_{x,t,y}\left[\int_t^{t+h} \ee^{\rho(T-s)}c(X^{u,\eta}_{s-},s,\eta_{s-}) \circ \dd\eta_s + \int_{t+h}^T \ee^{\rho(T-s)}c(X^{u,\eta}_{s-},s,\eta_{s-}) \circ \dd \xh_s\right]\nonumber\\
    && = \Eb_{x,t,y}[g(X^{u,\eta}_{(t+h)-},t+h,\eta_{(t+h)-}) + \gamma \operatorname{Var}_{X^{u,\eta}_{(t+h)-},t+h,\eta_{(t+h)-}}(\Xh_T)]\nonumber\\
   && + \gamma \operatorname{Var}_{x,t,y}[\Eb_{X^{u,\eta}_{(t+h)-},t+h,\eta_{(t+h)-}}(\Xh_T)]\nonumber
   \end{eqnarray}
   \begin{eqnarray}
   \label{2.5}
    &&+ \theta \Eb_{x,t,y}\left[\int_t^{t+h} \ee^{\rho(T-s)}c(X^{u,\eta}_{s-},s,\eta_{s-}) \circ \dd\eta_s + \int_{t+h}^T \ee^{\rho(T-s)}c(X^{u,\eta}_{s-},s,\eta_{s-}) \circ \dd \xh_s\right]\nonumber\\
&&  = \Eb_{x,t,y}[V(X^{u,\eta}_{(t+h)-},t+h,\eta_{(t+h)-})] + \theta \Eb_{x,t,y}\left[\int_t^{t+h} \ee^{\rho(T-s)}c(X^{u,\eta}_{s-},s,\eta_{s-}) \circ \dd\eta_s\right]\nonumber\\
       &&  + \gamma \Eb_{x,t,y}[g^2(X^{u,\eta}_{(t+h)-},t+h,\eta_{(t+h)-})-g^2(x-\Delta\eta_t,t,y+\Delta\eta_t)] \nonumber\\
        &&+ \gamma [(\Eb_{x,t,y}[g(X^{u,\eta}_{(t+h)-},t+h,\eta_{(t+h)-})])^2 -g^2(x-\Delta\eta_t,t,y+\Delta\eta_t)].
\end{eqnarray}
Applying It\^o's formula to $V$, $g$ and  $g^2$ respectively, and using Condition (4),  we obtain the followings:
\begin{equation}
\label{3}
    \begin{aligned}
       & \Eb_{x,t,y}[V(X^{u,\eta}_{(t+h)-},t+h,\eta_{(t+h)-})] - V(x-\Delta\eta_t,t,y+\Delta\eta_t) \\
    =&   \Eb_{x,t,y}\biggl[\int_t^{t+h} (\Ac^u V)(X^{u,\eta}_{s-},s,\eta_{s-})\dd s + \int_t^{t+h} [V_y(X^{u,\eta}_{s-},s,\eta_{s-}) - V_x(X^{u,\eta}_{s-},s,\eta_{s-})]\dd \eta^c_s  \\
    & \ \ \ \ + \sum_{t<s<t+h}\int_0^{\Delta\eta_s}[V_y(X^{u,\eta}_{s-}-z,s,\eta_{s-}+z) - V_x(X^{u,\eta}_{s-}-z,s,\eta_{s-}+z)\dd z]\biggl],
    \end{aligned}
\end{equation}

\begin{equation}
\label{4}
    \begin{aligned}
        &\Eb_{x,t,y}[g^2(X^{u,\eta}_{t+h-},t+h,\eta_{t+h-})] - g^2(x-\Delta\eta_t,t,y+\Delta\eta_t) \\
        =&  \Eb_{x,t,y} \biggl[\int_t^{t+h} (\Ac^u g^2)(X^{u,\eta}_{s-},s,\eta_{s-}) \biggl] + \int_t^{t+h}\left[(2g g_y)(X^{u,\eta}_{s-},s,\eta_{s-})-(2g g_x)(X^{u,\eta}_{s-},s,\eta_{s-})\right] \dd \eta^c_s \\
         & \ \ \ \ + \sum_{t<s<t+h}\int_0^{\Delta\eta_s} \left[(2g g_y)(X^{u,\eta}_{s-}-z,s,\eta_{s-}+z)-(2g g_x)(X^{u,\eta}_{s-}-z,s,\eta_{s-}+z)\right] \dd z,   \end{aligned}
\end{equation}
and 
\begin{equation}
    \begin{aligned}
       & \left|\Eb_{x,t,y}[g(X^{u,\eta}_{t+h-},t+h,\eta_{t+h-})]-g(x-\Delta\eta_t,t,y+\Delta\eta_t)\right|\\
      =& \left| \Eb_{x,t,y}\biggl[ \int_t^{t+h}(\Ac^u g)(X^{u,\eta}_{s-},s,\eta_{s-})\dd s + \int_t^{t+h} [g_y(X^{u,\eta}_{s-},s,\eta_{s-})-g_x(X^{u,\eta}_{s-},s,\eta_{s-})]\dd \eta_s^c\right.\\
      &\ \ \ \ +\left.\sum_{t<s<t+h}\int_0^{\Delta \eta_s} [g_y(X^{u,\eta}_{s-}-z,s,\eta_{s-}+z)-g_x(X^{u,\eta}_{s-}-z,s,\eta_{s-}+z)]dz\biggl] \right| \\
      \le&  \Eb_{x,t,y}\biggl[ \sup_{s\in (t,t+h)} |(\Ac^u g)(X^{u,\eta}_{s-},s,\eta_{s-})|h + \sup_{s\in (t,t+h)}| [g_y(X^{u,\eta}_{s-},s,\eta_{s-})-g_x(X^{u,\eta}_{s-},s,\eta_{s-})]|Mh \biggl].
    \end{aligned}
\end{equation}
Then 
\begin{equation}
\label{5}
    \begin{aligned}
       & (\Eb_{x,t,y}[g(X^{u,\eta}_{t+h-},t+h,\eta_{t+h-})])^2-g^2(x-\Delta\eta_t,t,y+\Delta\eta_t)\\
       =&\Eb_{x,t,y}[g(X^{u,\eta}_{t+h-},t+h,\eta_{t+h-})-g(x-\Delta\eta_t,t,y+\Delta\eta_t)]^2 \\
       & \ \ \ \ + 2g(x-\Delta\eta_t,t,y+\Delta\eta_t)\Eb_{x,t,y}[g(X^{u,\eta}_{t+h-},t+h,\eta_{t+h-})-g(x-\Delta\eta_t,t,y+\Delta\eta_t)]\\
        =& 2g(x-\Delta\eta_t,t,y+\Delta\eta_t)\Eb_{x,t,y}\bigg[ \int_t^{t+h}(\Ac^u g)(X^{u,\eta}_{s-},s,\eta_{s-})\dd s\\
       &\ \ \ \ + \int_t^{t+h} [g_y(X^{u,\eta}_{s-},s,\eta_{s-})-g_x(X^{u,\eta}_{s-},s,\eta_{s-})]\dd \eta_s^c)\\
      &\ \ \ \ +\sum_{t<s<t+h}\int_0^{\Delta \eta_s} [g_y(X^{u,\eta}_{s-}-z,s,\eta_{s-}+z)-g_x(X^{u,\eta}_{s-}-z,s,\eta_{s-}+z)]dz\bigg] +o(h).
    \end{aligned}
\end{equation}
Substituting \eqref{3}, \eqref{4} and \eqref{5} into \eqref{2.5} and using  
\begin{align*}
J(x,t,y;\ph,\xh) &= V(x,t,y)  = V(x-\Delta\eta_t,t,y+\Delta\eta_t)\\
&- \int_0^{\Delta \eta_t}[V_y(x-u,t,y+u)-V_x(x-u,t,y+u)]\dd u \\
& \le V(x-\Delta\eta_t,t,y+\Delta\eta_t) + \int_0^{\Delta \eta_t}\theta c(x-z,t,y+z)\ee^{\rho(T-t)} \dd z, 
\end{align*}
yield
{\small 
\begin{eqnarray}\label{J-J}
 &&J(x,t,y;u,\eta)  J(x,t,y;\ph,\xh)\nonumber\\
     &\ge&  \Eb_{x,t,y} \biggl\{ \int_t^{t+h}  \big[(\Ac^u V)(X^{u,\eta}_{s-},s,\eta_{s-})+ \gamma(\Ac^u g^2)(X^{u,\eta}_{s-},s,\eta_{s-})- \gamma (2g\Ac^u g)(X^{u,\eta}_{s-},s,\eta_{s-})\big] \dd s \nonumber\\
      && + \int_t^{t+h} \big[\theta c(X^{u,\eta}_{s-},s,\eta_{s-}) \ee^{\rho(T-s)}+V_y(X^{u,\eta}_{s-},s,\eta_{s-}) - V_x(X^{u,\eta}_{s-},s,\eta_{s-})\big]\dd \eta^c_s \nonumber\\
      & &+ \sum_{t<s<t+h}\int_0^{\Delta\eta_s}\big[\theta c(X^{u,\eta}_{s-}-z,s,\eta_{s-}+z) \ee^{\rho(T-s)} + V_y(X^{u,\eta}_{s-}-z,s,\eta_{s-}+z)\nonumber\\
      &&- V_x(X^{u,\eta}_{s-}-z,s,\eta_{s-}+z)\big]\dd z +  \gamma \int_t^{t+h}\big[(2gg_y)(X^{u,\eta}_{s-},s,\eta_{s-})-(2gg_x)(X^{u,\eta}_{s-},s,\eta_{s-})\big] \dd \eta^c_s \nonumber\\
       & &+\gamma\sum_{t<s<t+h}\int_0^{\Delta\eta_s} \left[(2g g_y)(X^{u,\eta}_{s-}-z,s,\eta_{s-}+z)-(2g g_x)(X^{u,\eta}_{s-}-z,s,\eta_{s-}+z)\right] \dd z\nonumber\\
       && -2  \gamma \Eb_{x,t,y}\biggl\{ \int_t^{t+h}[g(x-\Delta\eta_t,t,y+\Delta\eta_t) - g(X^{u,\eta}_{s-},s,\eta_{s-})](\Ac^u g)(X^{u,\eta}_{s-},s,\eta_{s-})\dd s \nonumber\\
          & &+ g(x-\Delta\eta_t,t,y+\Delta\eta_t)\biggl[  \int_t^{t+h} [g_y(X^{u,\eta}_{s-},s,\eta_{s-})-g_x(X^{u,\eta}_{s-},s,\eta_{s-})]\dd \eta_s^c\nonumber\\
      &&+\sum_{t<s<t+h}\int_0^{\Delta \eta_s} [g_y(X^{u,\eta}_{s-}-z,s,\eta_{s-}+z)-g_x(X^{u,\eta}_{s-}-z,s,\eta_{s-}+z)]\dd z\biggl] \biggl\}+o(h).
\end{eqnarray}}
Because 
\begin{equation}
    \begin{aligned}
        (\Ac^u V)(x,t,y)+\gamma(\Ac^u g^2-2g\Ac^u g)(x,t,y)  & = (\Ac^u V)(x,t,y) + \gamma \sigma^2(x,t,y,u(x,t,y))g^2_x(x,t,y)\\
        & \ge
(\Ac^{\Ph}V)(x,t,y)+ \gamma \sigma^2(x,t,y,\ph(x,t,y))g^2_x(x,t,y) \ge 0, 
    \end{aligned}
\end{equation}
using \eqref{min}, we obtain 
\begin{equation*}
\begin{aligned}
    & \limsup_{h \downarrow 0}\frac{1}{h}[J(x,t,y;u,\eta) - J(x,t,y;\ph,\xh)]\\
    &\ge  \limsup_{h \downarrow 0}2 \gamma\Eb_{x,t,y} \biggl[ \frac{1}{h}\int_t^{t+h}[g(X^{u,\eta}_{s-},s,\eta_{s-})-g(x-\Delta\eta_t,t,y+\Delta\eta_t)][g_y(X^{u,\eta}_{s-},s,\eta_{s-})-g_x(X^{u,\eta}_{s-},s,\eta_{s-})]\dd \eta_s^c \\
    & +\frac{1}{h}\sum_{t<s<t+h}\int_t^{t+h}[g(X^{u,\eta}_{s-},s,\eta_{s-})-g(x-\Delta\eta_t,t,y+\Delta\eta_t)][g_y(X^{u,\eta}_{s-}-z,s,\eta_{s-}+z)\\
    &-g_x(X^{u,\eta}_{s-}-z,s,\eta_{s-}+z)] \dd z - \frac{1}{h}\int_t^{t+h}[g(X^{u,\eta}_{s-},s,\eta_{s-})-g(x-\Delta\eta_t,t,y+\Delta\eta_t)](\Ac^u g)(X^{u,\eta}_{s-},s,\eta_{s-}) \dd s\biggl]\\
    &\ge  - 2 \gamma \Eb_{x,t,y} \biggl[\lim_{h \downarrow 0}\sup_{t<s<t+h}\left|[g(X^{u,\eta}_{s-},s,\eta_{s-})-g(x-\Delta\eta_t,t,y+\Delta\eta_t)][g_y(X^{u,\eta}_{s-},s,\eta_{s-})-g_x(X^{u,\eta}_{s-},s,\eta_{s-})]\right| \\
    &\times \limsup_{h \downarrow 0}\frac{1}{h}\int_t^{t+h}\dd\eta_s + \lim_{h \downarrow 0}\sup_{t<s<t+h}|[g(X^{u,\eta}_{s-},s,\eta_{s-})\\
    &-g(x-\Delta\eta_t,t,y+\Delta\eta_t)](\Ac^u g)(X^{u,\eta}_{s-},s,\eta_{s-})|\times \limsup_{h \downarrow 0} \frac{1}{h} \int_t^{t+h}\dd s\biggl]=0,
\end{aligned}
\end{equation*}\noindent
where the second inequality follows from Condition (4) and Dominated convergence theorem, the last equality follows from the fact that $X^{u,\eta}$ is a c\`adl\`ag process and that $\eta_{t+h-}-\eta_t \le Mh$ almost surely. Thus, $(\Ph,\Xih)$ is an equilibrium singular-regular control law.
\end{proof}

\begin{remark}
     Conditions (1) and (2) in Theorem~\ref{verification thm} can be relaxed. Indeed, define $\Qc_T:=\Rb \times [0,T) \times [0,m]$, if 
     $$V,g \in C^{2,1,1}( \overline{W^{\Xih}} \cap \Qc_T) \cap C^{2,1,1}(P\cap \Qc_T) \cap C^{1,1,1}(\Qc_T) \cap C(\Qc),$$ and  \eqref{min} holds separately in the waiting region $\overline{W^{\Xih}}$ and the action region $P^{\Xih}$. By using the change-of-variable formula with local time on surfaces (see Theorem 3.2 of \textcite{martin2010seminaire}), the only modification required in the proof is to replace the second-order derivative on the free boundary with the average of its left- and right-hand second-order derivatives due to $V,g \in C^{1,1,1}(\Qc)$. Moreover, when the free boundary exhibits certain special features, the smoothness requirement on the auxiliary function $g$ at the free boundary can be further relaxed; see Theorem~\ref{example 1} for details.
\end{remark}

\section{Necessary Conditions}
\label{necessary}
In Section~\ref{verfication}, we established a verification theorem providing sufficient conditions for equilibrium strategies. In this section, we investigate   the necessary conditions,under which any equilibrium pair must satisfy the associated extended HJB equations, that is, the following theorem.
\begin{theorem}
\label{necessary thm}
    Suppose that $(\Ph,\Xih)$ is an equilibrium regular-singular control law with equilibrium value function $V$. For any $(x,t,y) \in \Qc$, let $\Xh := X^{x,t,y;\Ph,\Xih}$ and $( \ph, \xh) :=( \pi^{x,t,y;\Ph,\Xih}, \xi^{x,t,y;\Ph,\Xih})$ be the state process and the regular-singular control pair, respectively. Define $g(x,t,y) = \Eb_{x,t,y}[\Xh_T]$. Assume the following conditions hold:
    
    (1) $V,g \in C^{2,1,1}(\Qc)$.
    \vspace{2mm}

    (2) For any $(x,t,y) \in \Qc$, there exists an $h_0>0$, such that for any admissible regular-singular control pair of perturbations $(u,\eta) \in \Dc'_{[t,t+h)}, \forall h \in (0,h_0)$, the corresponding admissible pair of controls, defined in \eqref{pi^h,xi^h}, still denote by $(u,\eta)$, for functions $\varphi = V,g,g^2, \phi = \Ac^ug, g_y-g_x, g$, it holds that 
\begin{equation}
    \label{finite2}
    \Eb_{x,t,y}[\sup_{s\in [t,t+h)}|\phi(X^{u,\eta}_s,s,\eta_s)|] + \Eb_{x,t,y}\left[ \int_t^{t+h} \varphi_x^2(X^{u,\eta}_s,s,\eta_s) \dd s\right] <+\infty.
\end{equation}
Then the extended HJB equations in Condition (2) of Theorem~\ref{verification thm} hold and $\Ph$ has the expression in \eqref{Ph_def}.
\end{theorem}

\begin{proof}
\textbf{Step 1.} {\bf  We prove equations \eqref{gW}, \eqref{gP} and  \eqref{gT} hold.} 

Fix arbitrary $(x,t,y)\in \Qc$ with $t<T$ and  $ \forall \ h \in (0,T-t)$. Applying It\^o's formula to $g$,  taking expectations and using Condition(2) yield 
\begin{equation}
\label{g1}
    \begin{aligned}
        0 & = \Eb_{x,t,y}[g(\Xh_{t+h-},t+h,\xh_{t+h}) - g(x,t,y)] \\
        & = \Eb_{x,t,y}\biggl[\int_t^{t+h}(\Ac^{\Ph}g)(\Xh_{s-},s,\xh_{s-}) \dd s + \int_t^{t+h} [g_y(\Xh_{s-},s,\xh_{s-})-g_x(\Xh_{s-},s,\xh_{s-})] \dd \xh_s^c \\
        & \quad + \sum_{s\in[t,t+h)} \int_0^{\Delta\xh_s}[g_y(\Xh_{s-}-u,s,\xh_{s-}+u)-g_x(\Xh_{s-}-u,s,\xh_{s-}+u)] \dd u \biggl].
    \end{aligned}
\end{equation}
Letting $h \downarrow 0$, we obtain $
\Eb_{x,t,y}[g(x-\Delta\xh_t,t,y+\Delta\xh_t)-g(x,t,y)]=0$,
then for all $m\in[0,\Delta\xh_t]$, $g(x-m,t,y+m) = g(x,t,y)$,
thus we have $$g_y(x,t,y)=g_x(x,t,y),\forall (x,t,y)\in P^{\Xih}.$$
For any $(x,t,y)\in W^{\Xih}$, there exists $h_1$ sufficiently small such that $\xh_s = \xh_t, \forall s\in[t,t+h_1]$, and thus we have from \eqref{g1} that $\Eb_{x,t,y} \left[\frac{1}{h}\int_t^{t+h}(\Ac^{\Ph}g)(\Xh_{s},s,\xh_{s}) \dd s =0\right]$. Then, using Condition (2) and the dominated convergence theorem, we have $$(\Ac^{\Ph}g)(x,t,y) = 0, \forall(x,t,y) \in W^{\Xih}.$$
The terminal condition follows easily from the probabilistic expression of $g$.

\textbf{Step 2.} {\bf We prove that both expressions within the curly braces in \eqref{min} are nonnegative and the terminal condition \eqref{VT} holds.}

Using the techniques in \textbf{Step 3} in the proof of Theorem~\ref{verification thm},  for any $(x,t,y) \in \Qc$ and any admissible regular-singular control pair of perturbation $(u,\eta) \in \Dc'_{[t,t+h)}, h < h_0$, where $h_0$ is defined in Condition (2), denoting the corresponding admissible  perturbed regular-singular  control pair still by $(u,\eta)$, we obtain 
\begin{equation}
    \begin{aligned}
     \label{1.1}
     &   J(x,t,y;u,\eta) - J(x,t,y;\ph,\xh)\\
     = &  \Eb_{x,t,y} \biggl\{ \int_t^{t+h}  [(\Ac^u V)(X^{u,\eta}_{s-},s,\eta_{s-})+\gamma \sigma^2(X^{u,\eta}_{s-},s,\eta_{s-},u(X^{u,\eta}_{s-},s,\eta_{s-})) (g_x^2)(X^{u,\eta}_{s-},s,\eta_{s-})] \dd s \\
      & + \int_t^{t+h} [\theta c(X^{u,\eta}_{s-},s,\eta_{s-}) \ee^{\rho(T-s)}+V_y(X^{u,\eta}_{s-},s,\eta_{s-}) - V_x(X^{u,\eta}_{s-},s,\eta_{s-})]\dd \eta^c_s \\
      & +\sum_{t\le s<t+h}\int_0^{\Delta\eta_s}[\theta c(X^{u,\eta}_{s-}-z,s,\eta_{s-}+z) \ee^{\rho(T-s)} + V_y(X^{u,\eta}_{s-}-z,s,\eta_{s-}+z) - V_x(X^{u,\eta}_{s-}-z,s,\eta_{s-}+z)] \dd z 
      \end{aligned}
      \end{equation}
      \begin{equation*}
      \begin{aligned}
      & + \gamma \int_t^{t+h}[2(g(X^{u,\eta}_{s-},s,\eta_{s-})-g(x,t,y))(g_y(X^{u,\eta}_{s-},s,\eta_{s-})-g_x(X^{u,\eta}_{s-},s,\eta_{s-})] \dd \eta^c_s\\
       & + \gamma \sum_{t\le s<t+h}\int_0^{\Delta\eta_s} \left[(2(g-g(x,t,y))g_y)(X^{u,\eta}_{s-}-z,s,\eta_{s-}+z)-(2(g-g(x,t,y)) g_x)(X^{u,\eta}_{s-}-z,s,\eta_{s-}+z)\right] \dd z\\
       & -2 \gamma \int_t^{t+h}[g(x-\Delta\eta_t,t,y+\Delta\eta_t) - g(X^{u,\eta}_{s-},s,\eta_{s-})](\Ac^u g)(X^{u,\eta}_{s-},s,\eta_{s-})\dd s \biggl\}+o(h).
    \end{aligned}
\end{equation*}
Using Definition~\ref{equ_law}, we have $\liminf\limits_{h\downarrow 0}\{J(x,t,y;u,\eta) - J(x,t,y;\ph,\xh)\} \le 0$ yielding $$
\int_0^{\Delta \eta_t}[\theta c(X^{u,\eta}_{t-}-z,t,\eta_{t-}+z) \ee^{\rho(T-t)}+V_y(x-z,t,y+z)-V_x(x-z,t,y+z)]\dd z \ge 0,$$
for arbitrary $\Delta\eta_t \in[0,m-y]$. Thus, we obtain $$\theta c(x,t,y) \ee^{\rho(T-t)}+V_y(x,t,y)-V_x(x,t,y) \ge 0.$$
Choosing $\eta_s = y,\forall s\in[t,T]$, using the fact that $\liminf\limits_{h \downarrow 0}\frac{J(x,t,y;u,\eta) - J(x,t,y;\ph,\xh)}{h} \ge 0$, Condition (2) and Dominated convergence theorem, we have $
(\Ac^u V)(x,t,y)+\sigma^2(x,t,y,u(x,t,y))(g_x^2)(x,t,y) \ge 0$
for arbitrary $u(x,t,y) \in U_1$, then
\begin{align}
\label{1.2}
    \min_{\pi\in U_1}(\Ac^{\pi} V)(x,t,y)+ \gamma \sigma^2(x,t,y,\pi)(g_x^2)(x,t,y) \ge 0.
    \end{align}
The terminal condition follows easily from the probabilistic expression of $V$.

\textbf{Step 3.} {\bf We prove that equality holds in \eqref{min} and $\Ph$ has the expression in \eqref{Ph_def}.}

Let $(\pi^h,\xi^h) = (\ph,\xh)$ in \eqref{1.1}, then for any $(x,t,y) \in W^{\Xih}$, there exists $h_1>0$ small enough such that $\xh_s = \xh_t, \forall s\in[t,t+h]$. Dividing both sides of \eqref{1.1} by $h$, then letting $h \downarrow 0$, using Condition (2) and Dominated convergence theorem, we obtain 
$$(\Ac^{\Ph} V)(x,t,y)+ \gamma \sigma^2(x,t,y,\Ph(x,t,y))(g_x^2)(x,t,y) = 0.$$
Combining with \eqref{1.2}, we obtain  $\Ph$ has the expression in \eqref{Ph_def}. Note that for any $(x,t,y)\in P^{\Xih}$, using the same argument as  in \textbf{Step 1}, we have $V(x,t,y) = V(x-a,t,y+a) + \theta c(x,t,y) \ee^{\rho(T-t)}a,\  \forall \ a\in [0,\Delta\xh_t]$, thus
$$\theta c(x,t,y) \ee^{\rho(T-t)} -V_x(x,t,y) + V_y(x,t,y) =0. $$
Noting that $W^{\Xih} \cup P^{\Xih} = \Qc$, we complete the proof.
\end{proof}

\section{Reinsurance Application and Structural Interaction Analysis}
\label{Example}
This section specializes   the  regular-singular equilibrium framework \eqref{equ_J}  to a reinsurance model and derives explicit equilibrium strategies under the verification theorem \ref{verification thm}. The main purpose is to reveal the interaction between the regular and singular controls. In contrast to standard singular control problems, the singular free boundary depends explicitly on the regular control feedback law, leading to a coupled equilibrium structure. The explicit solution further enables us to identify parameter regimes in which this coupling disappears.

\subsection{Insurance Model Specification}

We now illustrate the general framework established above by considering a specific example that an insurance company operating two lines of business with different risk-exposure characteristics. The corresponding drift and volatility parameters of the businesses satisfy
$$
0<a_H<a_L,\qquad \sigma_H>\sigma_L\ge0.
$$
The ordering above
is imposed to create a nontrivial allocation problem. It excludes the case where one business line uniformly dominates the other in terms of both expected exposure and variability, and thus guarantees the presence of a genuine trade-off in the insurer's allocation decision. Moreover, the aggregate risk exposure exhibits mean-reverting behavior with rate $b>0$.

The insurance company is allowed to dynamically allocate its underwriting between the two business lines. To this end, the insurer adopts a progressively measurable regular control process $$ \pi = \{\pi_t\}_{t \ge 0},  \, \pi_t \in [0,1],$$
where $\pi_t$ denotes the proportion of total risk exposure allocated to the high-volatility business segment at time $t$, while the remaining proportion $1-\pi_t$ is assigned to the low-volatility segment. Moreover, the insurer can irreversibly transfer its risk exposure to the reinsurer by paying a reinsurance premium $c(t,y)$, where $y$ denotes the accumulated reinsurance level. The cumulative reinsurance position is represented by a non-decreasing c\`adl\`ag process $\xi=\{\xi_t\}_{t\ge0}$ with $\xi_{0-}=y_0\in[0,m]$, where $m$ is the maximal reinsurance capacity.

We assume that the reinsurance level affects the high-risk business line through a linear distortion term $\alpha_2(y-m)$, so that its effective drift becomes $\pi_t\big(a_H-\alpha_2(y-m)\big).$   Given reinsurance premium $c(t,y) = c_0(t) + \alpha_1 c_1(t) y$, we consider the SDE of the insurer's exposure process $X^{\pi,\xi}$ under regular control $\{\pi_t\}_{t \ge 0}$ and singular control $\{\xi_t\}_{t \ge 0}$:
\begin{equation}
\begin{cases}
dX^{\pi,\xi}_{t}=[a_{L}(1-\pi_{t})+[a_{H}-\alpha_2 (\xi_{t-}-m)]\pi_t-bX^{\pi,\xi}_{t}]dt+[\sigma_{L}(1-\pi_{t})+\sigma_{H} \pi_{t}]dB_{t}-d\xi_{t},\,\, t\in(0,T],\\
X^{\pi,\xi}_{0-}=x_{0}.
\end{cases}
\end{equation}
The objective of the insurer is choosing a control pair $(\pi,\xi)$ to minimize
\begin{equation}
\label{J}
J(x,t,y;\pi,\xi):=\mathbb{E}_{x,t,y}X^{\pi,\xi}_{T}+\gamma \operatorname{Var}_{x,t,y}X^{\pi,\xi}_{T}+\theta \mathbb{E}_{x,t,y}\int_{t}^{T}c(r,\xi_{r-})e^{\rho(T-r)} \circ d\xi_{r}.
\end{equation}

To obtain a tractable equilibrium characterization and focus on the interaction between regular and singular controls, we impose the following assumptions throughout the remainder of this section.

\begin{assumption} \label{assump-solution} (1)   For any $t \in [0,T]$,   $c_0(t) >0,\,\, \alpha_1 c_1(t) \ge 0, \, \, \theta c_0(T) \ge 1$. (2) $\sigma_L = 0$. (3) $a_L-a_H \ge \alpha_2 m$.  
\end{assumption}
\begin{remark} Note that, based on \eqref{VT}, the equilibrium value function $V$ at the terminal time $T$ is 
\begin{align*}
    V(x,T,y) & = x - \max_{r \in [0,m-y]} \left[r - \theta \int_0^r (c_0(T) + \alpha_1 c_1(T) (y+z)) \dd z\right] \\
    & = x-\max_{r \in [0,m-y]} \left\{(1-\theta [c_0(T) + \alpha_1 c_1(T) y])r - \frac{ \theta \alpha_1 c_1(T)}{2}r^2\right\}.
\end{align*} 
If $\theta \alpha_1 c_1(T) = 0$, then $V(x,T,y) = x$.  If $\theta \alpha_1 c_1(T) > 0$, based on the relative position of the axis of symmetry $\frac{1-\theta c_0(T)}{\theta\alpha_1 c_1(T)}-y$ of the quadratic function in \(r\) (enclosed in braces in the preceding equation) with respect to \(0\) and \(m-y\), the terminal value function assumes distinct forms. Without loss of generality, we restrict our attention to the case where \(\frac{1-\theta c_0(T)}{\theta\alpha_1 c_1(T)}-y \le 0\). Condition (1) of Assumption~\ref{assump-solution} ensures that, it is not optimal in equilibrium to immediately exhaust any reinsurance capacity at time $T$ for any $y\in[0,m]$. Consequently, the terminal equilibrium is $V(x,T,y) = x$, which is a continuous terminal value function. Some other cases are relegated to the Appendix~\ref{B}. Condition (2) serves to simplify the calculations involved in our problem while simultaneously retaining its essential features. Condition (3) of Assumption \ref{assump-solution} ensures that under Condition (2), any line of business with a high drift rate must have a low diffusion rate; otherwise, the insurer would only engage in business with both lower drift and lower diffusion rates and would cease to purchase reinsurance. As the equilibrium policies are straightforward to derive, our analysis is restricted to the non-trivial scenarios that carry analytical significance.
\end{remark}

\subsection{Equilibrium Characterization}
Note that the terminal singular control depends only on accumulative reinsurance and time, for all $(x,t,y) \in \Bc :=\{(x,t,y)\in \Qc: \ph(x,t,y) \in (0,1)\}$, we make an ansatz that 
\begin{align}\label{ansatz}
\begin{cases}
   V^*(x,t,y)  = u_3(t)x+u_2(t)y^2+u_1(t)y+u_0(t),\\
   g^*(x,t,y)  = v_3(t)x+v_2(t)y^2+v_1(t)y+v_0(t).
\end{cases}
\end{align}
To illustrate the derivation of the explicit solution, we proceed under the standing assumption that all denominators arising in the analysis are strictly positive. This assumption is made purely for computational clarity and does not restrict the generality of the final results, as the corresponding conditions will be explicitly characterized in Theorem~\ref{example 1}. We focus on the condition of $V^*(x,T,y) = x $, the extended HJB equations become the followings:
  \begin{align}
       & \min\biggl\{ \min_{u \in [0,1]} \left\{u_3'(t)x +u_2'(t)y^2 +  u_1'(t)y + u_0'(t) + [a_L(1-u) + [a_H-\alpha_2(y-m)] u - bx]u_3(t) \right.\quad \, \nonumber \\
       & \hspace{4.5cm} + \left.\gamma\sigma_H^2 u^2 v_3^2(t)\right\},\, 
        \theta c(t) \ee^{\rho(T-t)}-u_3(t) +2 u_2(t) y+u_1(t) \biggl\} = 0,  \label{4min}\\
    &v_3'(t)x+v_2'(t)y^2+v_1'(t)y+v_0'(t)+[a_L(1-\ph(x,t,y))+[a_H-\alpha_2(y-m)] \ph(x,t,y)-bx]v_3(t)=0,\nonumber\\
    &\hspace{11.8cm}  \forall (x,t,y) \in \overline{W^{*}\cap \Bc} ,\, t<T,\label{4gW}\\
      &  v_1(t) = v_2(t),\,\ \ \ \ \forall (x,t,y) \in P^{*}\cap \Bc,\, t<T, \label{4gP}
    \end{align}
with the terminal conditions 
\begin{align}
    u_3(T) & = v_3(T) =  1,\label{u3v3}\\
    u_2(T) & = v_2(T) = u_1(T) = v_1(T) = u_0(T) = v_0(T) = 0,\label{uv210}
\end{align}
where $\ph$ is the equilibrium singular control. The regular-singular control law $\Xi^* = (W^{*},P^{*})$ has the following form:
\begin{align*}
    W^{*} \cap \Bc & = \{(x,t,y)\in \Bc:\theta c(t,y)\ee^{\rho(T-t)} - u_3(t)+2 u_2(t)y+u_1(t)>0 \text{ or } y=m\},\\
    P^{*} \cap \Bc & = \{(x,t,y)\in \Bc:\theta c(t,y)\ee^{\rho(T-t)} - u_3(t)+2 u_2(t)y+u_1(t)=0 , y<m\}.
\end{align*}
Note that $\ph$ is independent of $x$, using the inequalities \eqref{4min}, \eqref{4gW} and\eqref{u3v3} as well as the arbitrariness of $x \in \overline{W^{*}\cap \Bc}$, we obtain $$ u_3(t) = \ee^{-b(T-t)},\, v_3(t) = \ee^{-b(T-t)}.$$

Now we find the concrete form of the regular control $\ph$ as follows:

Because $ a_L>a_H + \alpha_2 m$, we have 
\begin{align}
\label{pi_equilibrium}
    \ph(x,t,y) & = \argmin_{u\in[0,1]} \{(a_H-\alpha_2(y-m)-a_L) u_3(t) u+\gamma v_3^2(t)\sigma_H^2 u^2\}\nonumber\\
    & =  \frac{(a_L-a_H+\alpha_2(y-m))e^{b(T-t)}}{2\gamma \sigma_H^2}, \ \ \    \forall (x,t,y) \in \Bc  = \{(x,t,y) \in \Qc: t > f(y)\},
\end{align}
where 
\begin{equation}\label{f(y)} 
f(y): =\min\left\{ \left(\frac{1}{b}\ln\frac{a_L-a_H+\alpha_2(y-m)}{2\gamma\sigma_H^2}+T \right), T\right\}.
\end{equation}
 Then, in the $\overline{W^{*} \cap \Bc}$ region, using \eqref{4min}, \eqref{4gW}, \eqref{uv210} and the  arbitrariness of of $y$,  we have
\begin{equation}
\begin{aligned}
    V^*(x,t,y) & = \ee^{-b(T-t)}x-\frac{(a_H-a_L-\alpha_2(y-m))^2(T-t)}{4\gamma \sigma_H^2} + \frac{a_L}{b}(1-\ee^{-b(T-t)}),\, \ \forall (x,t,y) \in \overline{W^{*}\cap \Bc},\\
    g^*(x,t,y) & = \ee^{-b(T-t)}x + \frac{(a_L-a_H+\alpha_2(y-m))^2}{2\gamma \sigma_H^2}(t-T) +\frac{a_L}{b}(1-\ee^{-b(T-t)}), \  \forall (x,t,y) \in \overline{W^*\cap \Bc},
    \end{aligned}
\end{equation}
where 
\begin{align}
    W^{*} \cap \Bc & = \{(x,t,y) \in \Bc:y > y^*(t) \textbf{ or } y= m\}, \quad
    P^{*} \cap \Bc & = \{(x,t,y) \in \Bc: y \le y^*(t),\, y<m \},
\end{align} 
and 
\begin{equation}
\label{y^*}
    \begin{aligned}
        y^*(t) & = \begin{cases}
            \min\left\{\frac{\theta c_0(t) \ee^{\rho(T-t)}-\ee^{-b(T-t)}-\frac{\alpha_2 (t-T)}{2\gamma \sigma_H^2}(a_H-a_L+\alpha_2 m)}{\frac{\alpha_2^2 (T-t)}{2\gamma\sigma_H^2}-\theta \alpha_1 c_1(t) \ee^{\rho(T-t)}},\ m\right\}, & \forall t \in [0,T),\\
            \lim_{t \rightarrow T} y^*(t), & t = T.
        \end{cases}
    \end{aligned}
    \end{equation}  
Now we come to the $\overline{W^{*}} \cap \Bc^c$ region, where $\Bc^c$ is the complement of $\Bc$ in $\Qc$. Noting that $\xh_s =\xh_{t-}$, $\ph(x,t,y) = 1$, we have
\begin{equation}
\begin{cases}
    \dd X_s = [a_H-\alpha_2(y-m)-b X_s]\dd s + \sigma_H \dd B_s, \, & s \in [t,f(y)],\\
    \dd X_s = \left(a_L-\frac{[a_L-a_H+\alpha_2 (y-m)]^2}{2\gamma \sigma_H^2} \ee^{b(T-s)} - b X_s\right)\dd s + \frac{(a_L-a_H+\alpha_2(y-m))\ee^{b(T-s)}}{2\gamma\sigma_H}\dd B_s, \, & s\in (f(y),T],\\
    X_{t-} = x.
\end{cases}
\end{equation}
and 
\begin{align}
    &X_T = x (\ee^{-b(T-t)}-\ee^{-b(T-f(y))})+ \frac{a_H-\alpha_2(y-m)}{b}(\ee^{-b(T-f(y))}-\ee^{-b(T-t)})  +\int_t^{f(y)}\sigma_H \ee^{-b(T-s)} \dd B_s \nonumber \\
    & + \left\{x \ee^{-b(T-f(y))}- \frac{[a_L-a_H+\alpha_2(y-m)]^2}{2\gamma\sigma_H^2}(T-f(y)) + \frac{a_L}{b}(1-\ee^{-b(T-f(y))}) + \int_{f(y)}^T \frac{a_L-a_H+\alpha_2(y-m)}{2\gamma \sigma_H} \dd B_s\right\} .\nonumber
    \end{align}
    Notice that the expression enclosed in the curly braces represents $X_T$ starting from the initial point $(x,f(y),y)$, we obtain 
    \begin{align*}
    V^*(x,t,y) = & V^*(x,f(y),y) + [\ee^{-b(T-t)}-\ee^{-b(T-f(y))}]x + \frac{a_H-\alpha_2(y-m)}{b}(\ee^{-b(T-f(y))}-\ee^{-b(T-t)}) \\
    & + \frac{\gamma \sigma_H^2}{2b}(\ee^{-2b(T-f(y))}-\ee^{-2b(T-t)}),\quad \forall (x,t,y) \in \overline{W^{*}} \cap \Bc^c,\\
    g^*(x,t,y) =& (\ee^{-b(T-t)}-\ee^{-b(T-f(y))})x + (\ee^{-b(T-f(y))}-\ee^{-b(T-t)}) \frac{a_H+\alpha_2(m-y)}{b} \\
    & +g^*(x,f(y),y),  \quad \forall (x,t,y) \in \overline{W^*}\cap \Bc^c.
  \end{align*}
Assume 
\begin{equation}\label{y1-increasing}
\partial_y[\theta\ee^{\rho(T-t)}c(t,y)-V^{*}_x(x,t,y)+V^{*}_y(x,t,y)] 
    = \theta \ee^{\rho(T-t)}\alpha_1c_1(t) +\frac{\alpha_2^2}{2b\gamma\sigma_H^2}\ln{\frac{(a_L-a_H+\alpha_2(y-m))}{2\gamma\sigma_H^2}}\ge  0.
\end{equation}
Hence
\begin{equation}
    W^{*} \cap \Bc^c  = \{(x,t,y) \in \Bc^c: y > y^{*1}(t) \textbf{ or } y = m\}, \quad
P^{*} \cap \Bc^c  = \{(x,t,y) \in \Bc^c: y \le y^{*1}(t), y < m\},
\end{equation}
and 
\begin{equation}
\label{y^*1}
\begin{aligned}
y^{*1}(t) & :=\inf\left\{y\in[0,m]:\frac{a_L-a_H+\alpha_2(y-m)}{2\gamma\sigma_H^2}\frac{\alpha_2}{b}\left[-1+ \ln \left(\frac{a_L-a_H+\alpha_2(y-m)}{2\gamma\sigma_H^2}\right)\right]\right.\\
&\quad\quad\quad\quad +\left(\frac{\alpha_2}{b}-1\right)\ee^{-b(T-t)}+\theta \ee^{\rho(T-t)}[c_0(t) +\alpha_1 c_1(t)y] > 0 \Biggl \},
\end{aligned}
\end{equation}
where $\inf \emptyset := m$.
Thus, for $(x,t,y)\in P^{*}$, we have $(x-(k(t)-y),t,k(t)) \in \overline{W^*}$, and
\begin{eqnarray}
  &&V^*(x,t,y) =  V^*(x-(k(t)-y),t,k(t))+\biggl[\theta (c_0(t)+\alpha_1 c_1(t)y)(k(t)-y) +\frac{\theta \alpha_1 c_1(t)}{2}(k(t)-y)^2]\biggl] \ee^{\rho(T-t)},\nonumber\\
 && g^*(x,t,y)  = g^*(x-(k(t)-y),t,k(t)),\nonumber\\
&&k(t) := \begin{cases}\label{k}
        y^*(t),\quad \forall t\in\{t\in[0,T]:y^*(t) \le \hat{y}(t)\},\\
        y^{*1}(t),\quad \forall t\in \{t\in[0,T]:y^{*1}(t) > \hat{y}(t)\}.
    \end{cases}
   \end{eqnarray}
The $ k(\cdot) $
   is the free boundary of the accumulated reinsurance, in which 
\begin{equation}
\label{hat y}
     \hat{y}(t)  : = \frac{a_H +\alpha_2m-a_L + 2\gamma\sigma_H^2\ee^{-b(T-t)}}{\alpha_2}.
\end{equation}
Indeed, $k(\cdot)$ denotes the switching boundary pseparating the region where the optimal regular control $\ph$ satisfies $\ph \equiv 1$ from the region where $\ph \in (0,1)$. \\
Summarizing the above results, we construct the candidate value function and the candidate auxiliary function as follows:
\begin{equation}
    \label{V*,g*}
    \begin{aligned}
       V^*(x,t,y) & = 
        \begin{cases}
             \ee^{-b(T-t)}x+\frac{(a_H-a_L-\alpha_2(y-m))^2}{4\gamma \sigma_H^2}(t-T) + \frac{a_L}{b}(1-\ee^{-b(T-t)}),\,  &\forall (x,t,y) \in \overline{W^*\cap \Bc},\\
             (\ee^{-b(T-t)} - \ee^{-b(T-f(y)})x + (\ee^{-b(T-f(y))}-\ee^{-b(T-t)})\frac{a_H+\alpha_2(m-y)}{b}\\
             \hspace{3cm}+(\ee^{-2b(T-f(y))}-\ee^{-2b(T-t)})\frac{\gamma \sigma_H^2}{2b}+V^*(x,f(y),y),\,&\forall (x,t,y) \in \overline{W^*}\cap \Bc^c, \\
            V^*(x-(k(t)-y),t,k(t))+\biggl[\theta (c_0(t)+\alpha_1 c_1(t)y)(k(t)-y)\\
           \hspace{5cm} +\frac{\theta \alpha_1 c_1(t)}{2}(k(t)-y)^2]\biggl] \ee^{\rho(T-t)},\, & \forall (x,t,y) \in P^*.
        \end{cases}
        \end{aligned}
        \end{equation}
        \begin{equation}
            \begin{aligned}
        g^*(x,t,y) & = 
        \begin{cases}
            \ee^{-b(T-t)}x + \frac{(a_L-a_H+\alpha_2(y-m))^2}{2\gamma \sigma_H^2}(t-T) +\frac{a_L}{b}(1-\ee^{-b(T-t)}), \hspace{2cm} \forall (x,t,y) \in \overline{W^*\cap \Bc},\\
            (\ee^{-b(T-t)}-\ee^{-b(T-f(y))})x + (\ee^{-b(T-f(y))}-\ee^{-b(T-t)}) \frac{a_H+\alpha_2(m-y)}{b}+g^*(x,f(y),y), \\
              \hspace{11.5cm} \forall (x,t,y) \in \overline{W^*}\cap \Bc^c,\\
            g^*(x-(k(t)-y),t,k(t)), \hspace{7.3cm} \forall  (x,t,y) \in P^*,
        \end{cases}
        \end{aligned}
    \end{equation} 
     \begin{equation}
     \label{W*P*}
     \begin{aligned}
       & W^*  := \{(x,t,y)\in \Qc:y > k(t)\},\quad P^*:=\{(x,t,y)\in \Qc:y \le k(t)\},\\
       & \Pi^*(x,t,y)  :=\min \left\{ \frac{[a_L-a_H+\alpha_2 (y-m)]}{2\gamma \sigma_H^2}\ee^{b(T-t)}, \, 1\right\} = \mathbf{1}_{\{y\ge\hat{y}(t)\}} + \frac{[a_L-a_H+\alpha_2 (y-m)]}{2\gamma \sigma_H^2}\ee^{b(T-t)}\mathbf{1}_{\{y < \hat{y}(t)\}},
     \end{aligned}
     \end{equation}
     and $y^*(t) ,y^{*1}(t), \hat{y}(t), k(t)$ are defined in \eqref{y^*}, \eqref{y^*1}, \eqref{hat y}, \eqref{k}.
     
 Because  the rigorous verification of the equilibrium strategies requires additional technical conditions, we refer to them as candidate functions at this stage. Furthermore, the regularity properties of these candidate functions are established in the following lemma.
\begin{lemma}
The candidate value function $V^*$ and the candidate auxiliary function $g^*$ defined in \eqref{V*,g*} have the following regularity properties:
    $$V^* \in C^{2,1,1}(\Qc),\ \mbox{and}\  g^* \in C^{2,1,1}(\overline{W^*}) \cap C^{2,1,1}(P^*)\cap C^{2,1,0}(\Qc).$$
\end{lemma}
\begin{proof}
It suffices to verify the regularity conditions across the free boundary $y=k(t)$ and the regular control switching boundary $y=\hat{y}(t)$. Indeed, by inspection of their explicit piecewise expressions, it is straightforward to see that both $V^*$ and $g^*$ are continuous functions over the entire domain $\mathcal{Q}$, and they are twice continuously differentiable with respect to $x$.

Indeed, for all $(x,t,y) \in \overline{W^*}\cap \Bc^c$, we have 
\begin{equation}
    \begin{aligned}
        V^*_y(x,t,y) &= [(\Ac^{\Ph^*}V)(x,f(y),y) + \gamma \sigma^2_H g_x^2(x,f(y),y)]f'(y) + V^*_y(x,f(y),y) - \frac{\alpha_2}{b} (\ee^{-b(T-f(y))}- \ee^{-b(T-t)}) \\
        & = V^*_y(x,f(y),y) - \frac{\alpha_2}{b} (\ee^{-b(T-f(y))} -\ee^{-b(T-t)}),\\
        V^*_t(x,t,y) & = -[a_H + \alpha_2(m-y)-bx]\ee^{-b(T-t)}-\gamma\sigma_H^2 \ee^{-2b(T-t)}. 
    \end{aligned}
\end{equation}
Then, as $(x,t,y)\rightarrow (x_1, f(y_1), y_1)$, note that $ (x,t,y)\in \overline{W^*}\cap \Bc^c, (x_1, f(y_1), y_1)\in \overline{W^*\cap \Bc^c}$, we have $V^*_y(x,t,y)\rightarrow V^*_y(x_1,f(y_1),y_1)$, $ V^*_t(x,t,y)\rightarrow V^*_t(x_1,f(y_1),y_1)$. By the construction of $V^*$, it is straightforward to verify that $V^*$ is continuously differentiable  with respect to $y$ across the free boundary $k(t)$. For all $(x,t,y) \in P^*$, we have 
\begin{equation*}
    \begin{aligned}
    V_t^*(x,t,y)  =& V^*_t(x-k(t)+y,t,k(t)) + [\theta c(t,y) \ee^{\rho(T-t)} - (V^*_x-V^*_y)(x-k(t)+y,t,k(t))]k'(t) \\
    & + \left[ \theta[c_0'(t) + \alpha_1 c_1'(t)y] + \frac{\theta\alpha_1 c_1'(t)}{2}(k(t)-y) + \theta\alpha_1c_1(t)k'(t)\right.\\
    &\left. - \rho \left(\theta c(t,y) + \frac{\theta\alpha_1 c_1(t)}{2}(k(t)-y)\right)\right]\ee^{\rho(T-t)}(k(t) - y)\\
     = & V^*_t(x-k(t)+y,t,k(t))  + \left[ \theta[c_0'(t) + \alpha_1 c_1'(t)y] + \frac{\theta\alpha_1 c_1'(t)}{2}(k(t)-y) + \theta\alpha_1c_1(t)k'(t)\right.\\
    & \left. -\rho \left(\theta c(t,y) + \frac{\theta\alpha_1 c_1(t)}{2}(k(t)-y)\right)\right]\ee^{\rho(T-t)}(k(t)-y)\\
    \rightarrow & V^*_t(x_1,t_1,k(t_1)) \quad  \text{ as },  (x,t,y) \rightarrow (x_1,t_1,k(t_1)).
    \end{aligned}
\end{equation*}
Thus $V^* \in C^{2,1,1}(\Qc)$. The regularity of the candidate auxiliary function $g^*$ can be established by a similar argument. Specifically, because  $g^*$ satisfies distinct differential equations within $W^*$ and $P^*$, respectively, and only the continuous pasting condition (rather than smooth pasting) is imposed during the solution procedure, a direct derivative comparison reveals that $g^*_y$ can be discontinuous across the free boundary $k(t)$.
\end{proof}

As $g_y$ may fail to be continuous at the free boundary $k(t)$, certain conditions must be imposed to guarantee that $V^*$ and $g^*$ indeed serve as the equilibrium value function and the associated auxiliary function, respectively. Therefore we propose the following theorem.

\begin{theorem}[Explicit equilibrium solution under no terminal jump]
\label{example 1} 
Under Assumption~\ref{assump-solution}, and assume the following conditions hold:

    (1) $\alpha_2 \neq 0$.
    
    (2) $\begin{aligned}[t]
        &\theta \alpha_1 c_1(t) \ee^{\rho (T-t)}- \frac{\alpha_2^2 (T-t)}{2\gamma \sigma_H^2} >0, \nonumber\\
        &\theta\alpha_1 c_1(t) \ee^{\rho (T-t)}+ \frac{\alpha_2^2}{2b\gamma\sigma_H^2}\ln{\frac{(a_L-a_H-\alpha_2 m)}{2\gamma\sigma_H^2}}\ge 0, \forall t \in[0,T].\nonumber
    \end{aligned}$
    
    \vspace{1em}
    (3) $k(T) \le 0,\, k'(t) \le 0,\,\forall t \in\{s \in [0, T]:k(t) >0\}$.
    \vspace{1em}
    
    (4) $\begin{aligned}[t]
        &  \left\{\frac{\dd [\theta\ee^{\rho(T-t)}c_0(t)]}{\dd t} +\frac{k(t)}{2}\frac{\dd [\alpha_1 c_1(t) \theta \ee^{\rho(T-t)}]}{\dd t} -b\ee^{-b(T-t)}\right\} \mathbf{1}_{\{k(t) \ge 0\}}\ge 0.\nonumber
    \end{aligned}$ 
    \vspace{1em}

    (5) $y^*(0) \ge \hat{y}(0)$,

   \noindent where $y^*(t)$, $y^{*1}(t)$, $\hat{y}(t)$ and $ k(t)$ are defined in \eqref{y^*}, \eqref{y^*1}, \eqref{hat y}, \eqref{k}, respectively.
    Then the regular-singular control law $(\Pi^*,\Xi^*)$ is an equilibrium law, where $\Xi^* = (W^*,P^*)$ and $W^*,P^*,\Pi^*$ are defined in \eqref{W*P*}, and  $V^*, g^*$  defined in \eqref{V*,g*} are the corresponding value function  and auxiliary function, respectively.
\end{theorem}

\begin{proof}
    \textbf{Step 1.} {\bf We show that $(V^*,g^*)$ satisfy the HJB equations in Theorem~\ref{verification thm}.}
    
Through a straightforward computation, it can be verified that  the pair $(V^*,g^*)$ satisfies \eqref{min}, $g^*$ satisfies \eqref{gW}  in $\overline{W^*}$ region, and  $g$ satisfies \eqref{gP}  in $P^*$ region. For all $(x,t,y) \in P^*$, we have 
\begin{equation}
    \label{V_t in P}
    \begin{aligned}
    & V^*_t(x,t,y)  = \left[\theta \ee^{\rho(T-t)}[c_0(t)+\alpha_1 c_1(t)k(t)] +  (V_y^*-V_x^*)(x-(k(t)-y),t,k(t))\right] k'(t)\\
    & + V^*_t(x-(k(t)-y),t,k(t)) +\frac{\dd [\theta \ee^{\rho(T-t)}(c_0(t) + \alpha_1 c_1(t)y)]}{\dd t}(k(t)-y) + \frac{\dd[\theta \ee^{\rho(T-t)}\alpha_1 c_1(t)]}{\dd t}\frac{(k(t)-y)^2}{2}.
    \end{aligned}
\end{equation}
Define $a_u(t) = a_L(1-u) + [a_H-\alpha_2(k(t)-m)]u, a^*(x,t,y) = a_L (1-\Pi^*(x,t,y)) + [a_H-\alpha_2(y-m)]\Pi^*(x,t,y)$. Observe that $(x-(k(t)-y),t,k(t)) \in \overline{W^*}\cap P^*$, we have
\begin{align}
    &\theta \ee^{\rho(T-t)}[c_0(t)+\alpha_1 c_1(t)k(t)]+(V^*_y-V^*_x)(x-(k(t)-y),t,k(t))=0,\label{PcapW1}\\
    &V^*_t(x-(k(t)-y),t,k(t)) = -\min_{u\in[0,1]}\left\{[a_u(t)-b(x-(k(t)-y))]\ee^{-b(T-t)}+\gamma\sigma_H^2u^2\ee^{-2b(T-t)}\right\}.\label{PcapW2}
    \end{align}
Substituting \eqref{PcapW1},\eqref{PcapW2} into \eqref{V_t in P}, we have 
\begin{equation*}
    \begin{aligned}
        & V^*_t(x,t,y) = -\min_{u\in[0,1]}\left\{[a_u(t)-b(x-(k(t)-y))]\ee^{-b(T-t)}+\gamma\sigma_H^2u^2\ee^{-2b(T-t)}\right\}\\
        & \hspace{5cm} +\frac{\dd [\theta \ee^{\rho(T-t)}(c_0(t) + \alpha_1 c_1(t)y)]}{\dd t}(k(t)-y) + \frac{\dd[\theta \ee^{\rho(T-t)}\alpha_1 c_1(t)]}{\dd t}\frac{(k(t)-y)^2}{2}.
    \end{aligned}
\end{equation*}
Thus we obtain 
\begin{eqnarray*}
           && (\Ac^{\Pi^*}V^*)(x,t,y) +\gamma\sigma_H^2(\Pi^*(x,t,y))^2 (g^*_x)^2(x,t,y)\\
           &=&  [a^*(x,t,y)-bx]\ee^{-b(T-t)} +\gamma \sigma_H^2 (\Pi^*(x,t,y))^2 \ee^{-2b(T-t)} +\frac{\dd [\theta \ee^{\rho(T-t)}(c_0(t) + \alpha_1 c_1(t)y)]}{\dd t}(k(t)-y) \\
           &&+ \frac{\dd[\theta \ee^{\rho(T-t)}\alpha_1 c_1(t)]}{\dd t}\frac{(k(t)-y)^2}{2}-\min_{u\in[0,1]}\left\{[a_u(t)-b(x-(k(t)-y))]\ee^{-b(T-t)}+\gamma\sigma_H^2u^2\ee^{-2b(T-t)}\right\}\\
        &\ge& [a^*(x,t,y)-bx]\ee^{-b(T-t)} +\gamma \sigma_H^2(\Pi^*(x,t,y))^2\ee^{-2b(T-t)}+\frac{\dd [\theta \ee^{\rho(T-t)}(c_0(t) + \alpha_1 c_1(t)y)]}{\dd t}(k(t)-y)\\
        &&+ \frac{\dd[\theta \ee^{\rho(T-t)}\alpha_1 c_1(t)]}{\dd t}\frac{(k(t)-y)^2}{2}-\left\{[a_{\Pi^*(x,t,y)}(t)-b(x-(k(t)-y))]\ee^{-b(T-t)}+\gamma\sigma_H^2 (\Pi^*(x,t,y))^2 \ee^{-2b(T-t)}\right\}\\
        &=&\left[(\alpha_2 \Pi^*(x,t,y)-b)\ee^{-b(T-t)}+\frac{\dd [\theta \ee^{\rho(T-t)}c_0(t)]}{\dd t}\right](k(t)-y) + \frac{\dd[\theta \ee^{\rho(T-t)}\alpha_1 c_1(t)]}{\dd t}\frac{k^2(t)-y^2}{2} \ge 0,
        \end{eqnarray*}
    combining with $\theta \ee^{\rho(T-t)} (c_0(t) + \alpha_1 c_1(t) y) - V^*_x(x,t,y) - V^*_y(x,t,y)  = 0$,
    we have the pair $(V^*,g^*)$ satisfies \eqref{min} in $P^*$ region.
\vskip 2pt
    \textbf{Step 2.} {\bf We prove $(\Pi^*,\Xi^*)$ is an equilibrium law  by demonstrating that the proof of Theorem~\ref{verification thm} remains valid under the current setting.} 
    \vskip 2pt 
    
Denote the singular control corresponding to $(\Pi^*,\Xi^*)$ by $\xi^*$.
  By using \textbf{Step 1} and Proposition~\ref{square int}, Conditions (2) and (3) of Theorem~\ref{verification thm} are satisfied. For Condition (4) of Theorem~\ref{verification thm}, if the initial time is $t$ and $X^{\Pi^*,\Xi^*}_{t-}=x$, $\xi^*_{t-}=y$, using $V^*_x$, $g^*_x$, $ 2g^*g^*_x$, $\Ac^u g^*$ ,$g^*_y-g^*_x$, $g^*$ are continuous with respect to $(t,y)$ separately in $\overline{W^{\Xih}}$ and $P^{\Xih}$ and have at most linear growth in $x$ and Proposition~\ref{square int}, then
  we have $\Eb_{x,t,y}[\sup_{s\in[t,T]}(X^{u,\eta}_s)^2] < + \infty$ for any admissible pair of perturbations $(u,\eta) \in \Dc'_{[t,t+h)}, \forall h \in (0,h_0)$, thus Condition (4) of Theorem~\ref{verification thm} holids.
  
Now we come to discuss Condition (1) of Theorem~\ref{verification thm}. Note that $V^* \in C^{2,1,1}(\Qc), g^* \in C^{2,1,1}(\overline{W^*}) \cap C^{2,1,1}(P^*)\cap C^{2,1,0}(\Qc)$ and it can be observed that $g^*_y$ is discontinuous solely at the points where $y = y_1(t)$. We come to solve the problems in  Theorem~\ref{verification thm} caused by the discontinuous of $g_y$ at the points of $y=k(t)$. In \textbf{Step 1} and \textbf{Step 2} of Theorem~\ref{verification thm}, note that $k(t)$ is decreasing with respect to $t$, therefore the purchase of the reinsurance can only occur at the initial time. Once the state-time-control triple enters $\overline{W^*}$, the accumulated reinsurance coverage remains constant. Thus for all  $(x,t,y) \in \Qc,$ we have $ g(x,t,y) = g(x-(\max\{y,k(t)\}-y),t,\max\{y,k(t)\}-y))$, with $(x-(\max\{y,k(t)\}-y),t,\max\{y,k(t)\}-y)) \in \overline{W^*}$, then we can applying  It\^o's formula over the interval from the completion of the initial decision to time $T$. In \textbf{Step 3} of Theorem~\ref{verification thm}, by the expression of $(W^*,P^*)$  and note that $k(t)$ is decreasing with respect to $t$, we deduce that, the path the state-time-control triple moves from $(x-\Delta\eta_t,t,y+\Delta\eta_t)$ to $(X^{u,\eta}_{t+h-},t+h,\eta_{t+h-})$ is within $\overline{W^*}$ or $P^*$ for sufficiently small $h$, hence if $V^*,g^* 
     \in C^{2,1,1}(\overline{W^{*}}) \cap C^{2,1,1}(P^{*})$, the It\^o's formula is still valid. Therefore, with the modifications discussed above, the same argument used in the proof of Theorem~\ref{verification thm} holds. Thus we complete the proof.    
\end{proof}

\begin{remark}
\label{c1 c_2 ex1}
    We illustrate the validity of the conditions in Theorem~\ref{example 1} by  an example as follows. If 
    $$
    \theta c_0(t) \ee^{\rho(T-t)} = \ee^{-b(T-t)},\, \,\,\,\, \theta \alpha_1 c_1(t) \ee^{\rho(T-t)} = 1+\frac{\alpha_2^2(t-T)}{2\gamma \sigma_H^2},$$
    a straightforward calculation shows that when $\alpha_2$ is sufficiently small, we have
    \begin{align*}
        &(y^*)'(t) = -\alpha_2\left( \frac{a_L-a_H-\alpha_2 m}{2\gamma \sigma_H^2}\right)\left(\frac{\alpha_2^2(T-t)}{\gamma \sigma_H^2}-1\right)^{-2} \propto -\alpha_2,\quad y^*(T)=0,\\
       & \frac{\alpha_2}{b}\frac{a_L-a_H+\alpha_2(y-m)}{2\gamma\sigma_H^2}\left[\ln \left(\frac{a_L-a_H+\alpha_2(y-m)}{2\gamma\sigma_H^2}\right)-1\right]+\frac{\alpha_2}{b}+y \ge y, \, \, \Longrightarrow \,\, y^{*1}(T) = 0,
    \end{align*}  
    and  for $ t \in \{s\in[0,T]: y^{*1}(s) \in (0,m),y^*(t) \ge \hat{y}(t)\}$, we have 
    \begin{equation}
    \label{y*1'}
    (y^{*1})'(t)  =\frac{(b-\alpha_2)\ee^{-b(T-t)}-\frac{\partial[\theta \ee^{\rho(T-t)}c(t,y^{*1})]}{\partial t}}{\theta\ee^{\rho(T-t)}\alpha_1c_1(t)+\frac{\alpha_2^2}{2b\gamma\sigma_H^2}\ln\left(\frac{a_L-a_H+\alpha_2(y^{*1}(t)-m)}{2\gamma\sigma_H^2}\right)} = \frac{-\alpha_2 \ee^{-b(T-t)}- \frac{ \alpha_2^2}{2\gamma\sigma_H^2} y^{*1}(t)}{1-\frac{\alpha_2^2(T-t)}{2\gamma\sigma_H^2}+\frac{\alpha_2^2}{2b\gamma \sigma_H^2}\ln{\frac{a_L-a_H+\alpha_2(y^{*1}(t)-m)}{2\gamma \sigma_H^2}}} \propto -\alpha_2.
    \end{equation}  
   Therefore, we have $k'(t) \propto -\alpha_2<0, k(T)=0$. Here, the symbol $\propto$ denotes proportionality. Then  the Condition (1)-(4) of Theorem~\ref{example 1} and Assumption~\ref{assump-solution} are satisfied  if $\alpha_2$ is sufficiently small, and since $$
   y^*(0) = \frac{\alpha_2 T(a_L - a_H -\alpha_2m)}{2\gamma\sigma_H^2-2\alpha_2^2T}>0>\frac{a_H+\alpha_2 m -a_L+2\gamma\sigma_H^2\ee^{-bT}}{\alpha_2} = \hat{y}(0),$$
    whenever $\alpha_2$ is sufficiently small and $T$ such that $-\frac{1}{b}\ln \frac{a_L-a_H-\alpha_2 m}{2\gamma\sigma_H^2}<T<\frac{\gamma\sigma_H^2}{\alpha_2^2}$, Condition (5) of Theorem~\ref{example 1} follows.
    
\end{remark}
\begin{remark}
    The quadratic ansatz in the state variable y is valid only in the regime where the regular control satisfies $\pi^* \in (0,1)$. When the control becomes binding, i.e., $\pi^* \equiv1$, this quadratic structure no longer holds, and the value function is governed by a different regime-specific representation. Consequently, the two regimes are associated with distinct free boundaries. This difference reflects that the waiting region, determined by the regular control, directly influences the structure of the singular control boundary, thereby inducing a regime-dependent equilibrium partition.
\end{remark}

The explicit equilibrium solution shows that the free boundary exhibits a regular-regime-dependent piecewise representation, which mirrors the switching dynamics of the equilibrium regular control. Naturally, it remains a fundamental question whether this piecewise boundary maintains its smoothness across the transition horizon. The following proposition demonstrates that the free boundary achieves continuous differentiability at the intersection of the regular switching boundary and the equilibrium free boundary.

\begin{proposition}
\label{C1}
Under Assumption~\ref{assump-solution} and the conditions of Theorem~\ref{example 1} holds, if the free boundary $k(\cdot)$ defined in ~\eqref{k} and the regular switching boundary $\hat{y}(\cdot)$ defined in ~\eqref{hat y} intersect at $t_0 \in (0,T)$ when $\hat{y}(0) \ge 0$.  Then the equilibrium free boundary $k(\cdot)$ is continuously differentiable at $t_0$, namely,
$$
k(t_0-)=k(t_0+),
\qquad
k'(t_0-)=k'(t_0+).
$$
Consequently, the regime transition of the regular control does not induce a turning point in the equilibrium free boundary.
\end{proposition}
\begin{proof}
    At time $t_0$, we have $y^*(t_0) = \hat{y}(t_0) := y_0$, which implies
    \begin{equation} \label{t_0}
        -\alpha_2 \mathrm{e}^{-b(T-t_0)}(T-t_0) - \mathrm{e}^{-b(T-t_0)} + \theta \mathrm{e}^{\rho(T-t_0)} \bigl[ c_0(t_0) + \alpha_1 c_1(t_0) y_0 \bigr] = 0.
    \end{equation}
    Substituting $(y_0, t_0)$ into the bracketed function in \eqref{y^*1} and combining it with \eqref{t_0}, we obtain
    \begin{equation}
        \begin{aligned}
          & \frac{a_L-a_H+\alpha_2(y_0-m)}{2\gamma\sigma_H^2} \frac{\alpha_2}{b} \left[ -1 + \ln \left( \frac{a_L-a_H+\alpha_2(y_0-m)}{2\gamma\sigma_H^2} \right) \right] + \left( \frac{\alpha_2}{b} - 1 \right) \mathrm{e}^{-b(T-t_0)} \\
          & + \theta \mathrm{e}^{\rho(T-t_0)} \bigl[ c_0(t_0) + \alpha_1 c_1(t_0) y_0 \bigr] = 0.   
        \end{aligned}
    \end{equation}
Then it follows from \eqref{y1-increasing} that $y^{*1}(t_0) = y_0 = y^*(t_0)$, which yields $k(t_0-) = k(t_0+)$.
    
    To proceed, recalling that $y^*(t)$ is implicitly defined as the unique root of the equation $F(t, y) = 0$, where
    \[
        F(t, y) := \theta \mathrm{e}^{\rho(T-t)} c_0(t) + \theta\alpha_1 \mathrm{e}^{\rho(T-t) } c_1(t) y - \mathrm{e}^{-b(T-t)} + \alpha_2 \frac{a_H-a_L-\alpha_2(y-m)}{2\gamma\sigma_H^2} (T-t),
    \]
   by using the Implicit Function Theorem, the derivative $(y^*)'(t)$ can be computed via $(y^*)'(t) = -\frac{\partial F / \partial t}{\partial F / \partial y}$, which leads to
    \begin{equation}
    \label{y*'}
        (y^*)'(t) = \frac{b\mathrm{e}^{-b(T-t)} - \frac{\alpha_2 [a_L-a_H+\alpha_2(y-m)]}{2\gamma\sigma_H^2} - \theta \bigl[ c_0(t) \mathrm{e}^{\rho(T-t)} \bigr]' - \theta \bigl[ \alpha_1 c_1(t) \mathrm{e}^{\rho(T-t)} \bigr]' y^*(t)}{\theta \alpha_1 c_1(t)\mathrm{e}^{\rho(T-t)} - \frac{\alpha_2^2(T-t)}{2\gamma\sigma_H^2}}.
    \end{equation}
    Utilizing the facts that $y^*(t_0) = y^{*1}(t_0) = \hat{y}(t_0) = y_0$ and \eqref{y*1'}, and substituting $(y_0, t_0)$ into \eqref{y*'}, we arrive at $(y^*)'(t_0) = (y^{*1})'(t_0)$, which implies $k'(t_0-) = k'(t_0+)$.
\end{proof}
\begin{remark}
\label{c1}
    Using Proposition~\ref{C1}, together with the expression of $k(t)$, we can deduce that the free boundary $k(t)$ is continuously differentiable whenever it is positive.
\end{remark}

\subsection{Equilibrium Space Partition and Interaction Mechanism}

The equilibrium strategy is jointly characterized by the regular control switching boundary $\hat{y}(\cdot)$ defined in \eqref{hat y} and the  free boundary $k(\cdot)$. Their mutual interaction uncovers a profound coupling mechanism, inducing a non-trivial partition of the equilibrium state space on $\Qc$.

Recall that the equilibrium regular control law admits a threshold structure of the form
\begin{equation}
  \Pi^*(x,t,y)=
  \begin{cases}
  \frac{a_L-a_H+\alpha_2(y-m)}{2\gamma\sigma_H^2}e^{b(T-t)}, & y \le \hat{y}(t),\\
  1, & y > \hat{y}(t).
  \end{cases}
\end{equation}
 Hence, the state space is naturally divided into the interior-control regime ($y \le \hat{y}(t)$) and the saturated-control regime ($y > \hat{y}(t)$). 
 
 On the other hand,   the equilibrium singular control is characterized by the free boundary $k(\cdot)$, which induces the geometric decomposition
$$
W^*:=\{(x,t,y)\in\Qc: y > k(t)\}, 
\qquad 
P^*:=\{(x,t,y)\in\Qc: y \le k(t)\},
$$
where $W^*$ and $P^*$ denote the waiting and action regions, respectively. In the waiting region $W^*$, no singular adjustment is undertaken and the cumulative reinsurance level remains static, while in the action region $P^*$, an instantaneous adjustment is executed to project the state onto the free boundary. In particular, although the free boundary $k(t)$ is formally included in the candidate action region derived from the equilibrium verification argument, no singular control is triggered when the state lies exactly on the free boundary. When the singular control $\xi$ satisfies $\xi_{t-} = k(t)$, the cumulative reinsurance level already attains its equilibrium target and no further singular intervention is required. Because the free boundary $k(t), t\in [0,T]$ is monotonically decreasing in time, 
which implies that the state-time-singular control triple remains in the waiting region or on the free boundary if no further singular control is exercised. Therefore, the optimal singular control policy acts only when the state-time-singular control triple lies strictly below the free boundary, in which case an instantaneous adjustment is undertaken to move the triple onto the free boundary.

As both $\hat{y}(t)$ and $k(t)$ are independent of the capital wealth state $x$, the partition structure is entirely invariant along the $x$-direction. Therefore, the three-dimensional equilibrium space can be perfectly projected onto the two-dimensional $(t,y)$-plane without loss of qualitative insights, as illustrated in Figures~\ref{fig:VT=x}-~\ref{fig:t0}.

\begin{figure}[htp]
    \centering
    \begin{subfigure}[b]{0.48\textwidth}
        \centering
        \includegraphics[width=\textwidth]{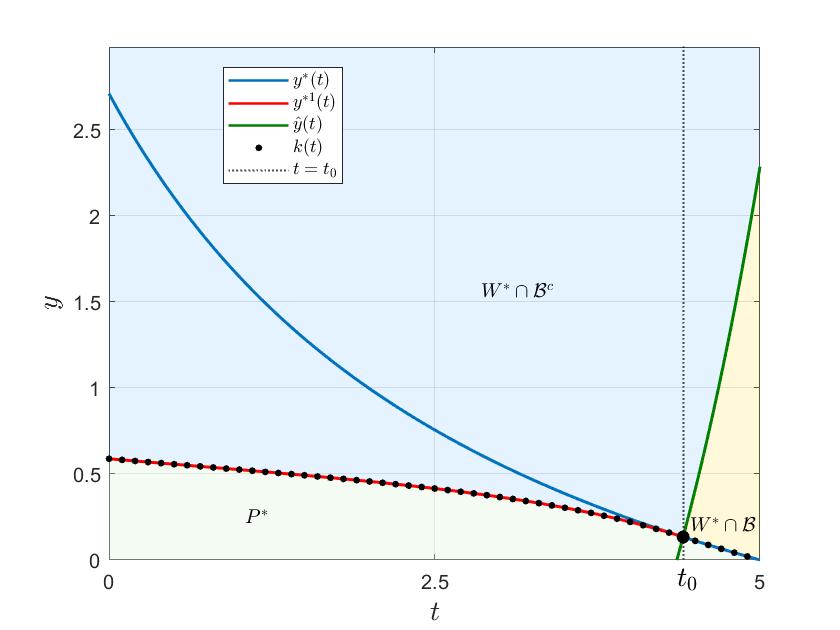}
        \caption{$\hat{y}(T) > y^*(T)$}
        \label{fig:partition1}
    \end{subfigure}
    \hfill
    \begin{subfigure}[b]{0.48\textwidth}
        \centering
        \includegraphics[width=\textwidth]{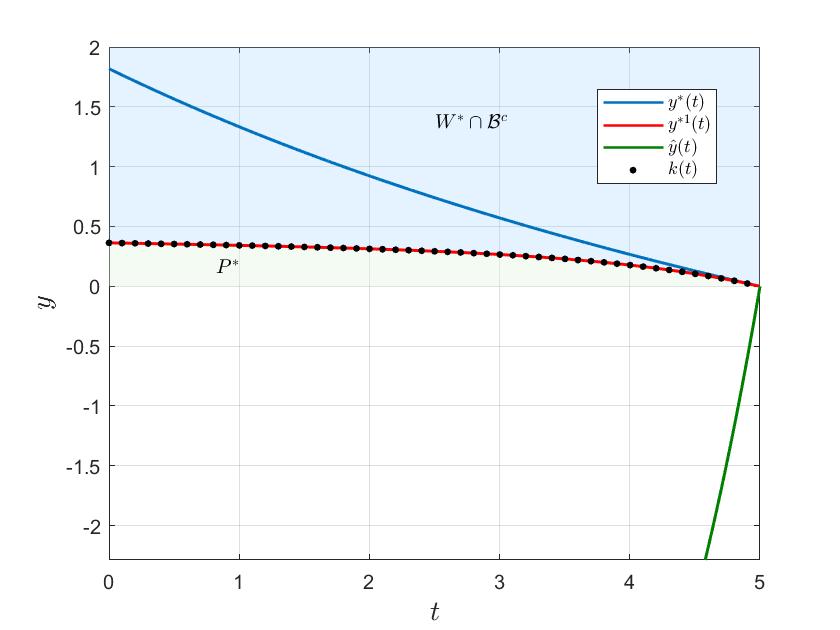}
        \caption{$\hat{y}(T) \le y^*(T)$}
        \label{fig:partition2}
    \end{subfigure}
    \caption{Space partitions jointly induced by the regular switching boundary $\hat y(t)$ and the singular free boundary $k(t)$ under $V^*(x,T,y) = x$.}
    \label{fig:VT=x}
\end{figure}

As captured in Figure~\ref{fig:VT=x}, when $\hat{y}(T) > k(T)$ (Figure 1(a)), the waiting region $W^*$ undergoes a non-trivial partition into an interior regime subregion $W^* \cap \mathcal{B}$ and a saturated regime subregion $W^* \cap \mathcal{B}^c$. This graphical structure clearly demonstrates that the active regular operational regime depends on the singular control state. This dependence stems  from the feedback loop of the time-inconsistent equilibrium system, where the regular control admits the representation $\pi^*_t = \Pi^*(X^{\pi^*,\xi^*}_{t-}, t, \xi^*_{t-})$. Because  the accumulated singular control $\xi^*$ enters the non-local extended HJB system, the equilibrium value function inherits this dependence. Consequently, the gradient condition defining the optimal stopping threshold changes when the system switches between the interior regime ($\pi^*\in(0,1)$) and the saturated regime ($\pi^*\equiv 1$), yielding a regime-dependent, piecewise representation of the equilibrium free boundary:
$$
k(t)=
\begin{cases}
y^*(t), & \forall t \in \{s\in[0,T]: y^*(s) \le \hat y(s)\},\\
y^{*1}(t), & \forall s \in \{t \in [0,T]: y^*(s) > \hat y(s)\}.
\end{cases}
$$
Despite this abrupt structural transition in the underlying control laws, the equilibrium free boundary maintains rigorous geometric smoothness across the switching epoch, as established in Proposition~\ref{C1}. When $\hat{y}(T) \le y^*(T)$, the equilibrium regular control remains constantly equal to $1$, and the free boundary is given by $k(t) = y^{*1}(t)$. In this case, the switching behavior of the free boundary vanishes.

\begin{figure}[htp]
\centering
\includegraphics[width=0.6\textwidth]{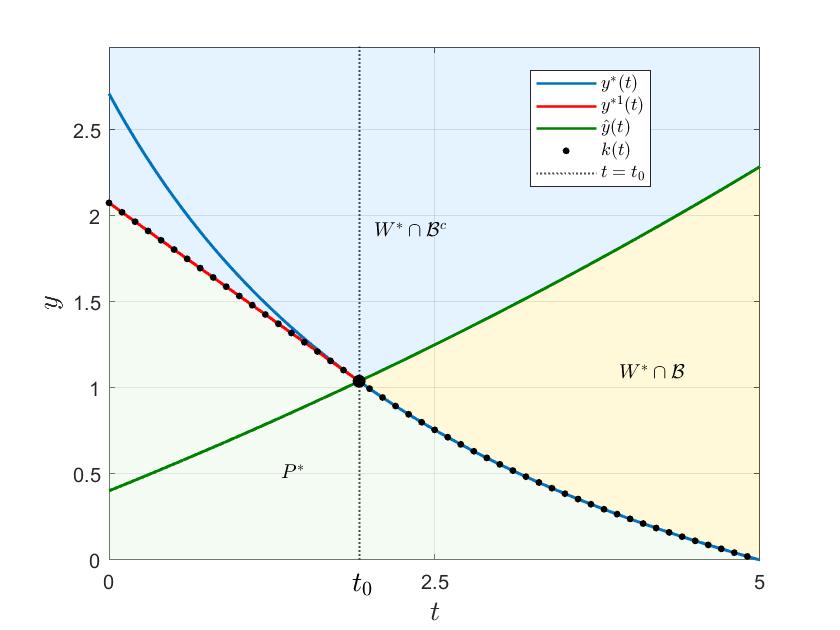}
\caption{Time-singular control space partition exhibiting a structural transition at the intersection epoch $t_0$.}
\label{fig:t0}
\end{figure}

It is worth noting that if the free boundary $k(t)$ and the regular switching boundary $\hat{y}(t)$ intersect at $t_0 \in (0, T)$ when $\hat{y}(0) \ge 0$, it will induce a regime-dependent structural break in the equilibrium policy. As shown in Figure~\ref{fig:t0}, the intersection time $t_0$ marks a structural transition in the equilibrium policy. For $t < t_0$, the candidate impulse target $y^*(t)$  characterized under the interior regular control regime ($\pi \in (0, 1)$) satisfies $y^*(t) > \hat{y}(t)$, which implies that any non-zero singular intervention instantaneously projects the state into the saturated regime $\pi \equiv 1$, rather than remaining in the interior regime. This global regime shift fundamentally alters the analytical form of the continuation value function, thereby destroying the invariance of the interior-regime candidate target. This jump changes the governing continuation value function because the post-intervention dynamics follow an entirely different control law. Consequently, the free boundary associated with the interior regime is no longer self-consistent before $t_0$, and must be endogenously redefined in terms of the continuation value under the saturated regime. 

In this sense, the equilibrium singular control induces a single, regime-switching projection, after which the relevant free boundary is endogenously determined by the new (post-jump) control structure. Consequently, the effective free boundary cannot be well-defined within a single local regime; rather, it is globally determined by the continuation value associated with the post-intervention saturated regular control law. 

Therefore, the regular and singular controls cannot be analyzed independently. The singular control state variable influences the equilibrium regular control strategy through the feedback representation, while the resulting regular control policy affects the characterization of the singular control boundary. This interaction reveals a mutually coupled structure between the regular and singular controls at the equilibrium level.

\subsection{Model Reduction and Decoupled Case}

In this subsection, we first derive the equilibrium strategies under two benchmark cases corresponding to (i) a model with only regular control and (ii) a model with only singular control. We then study the degenerate case $\alpha_2=0$, which eliminates the coupling between the two controls. It is shown that in this case the equilibrium strategies are completely decoupled and reduce to those obtained in the single-control problems. This demonstrates that the coupled equilibrium structure degenerates to the classical independent-control form when the coupling parameter equals zero.

We first analyze the degenerate cases involving only a single control.
\begin{theorem}
\label{singular_only}
    Under Assumption~\ref{assump-solution}, let $u \in [0,1]$ be given and fixed the  regular control by $\pi_t \equiv u, t\in[0,T]$, consider the reduced singular-control-only equilibrium problem, that is, choosing a singular control $\xi$ to minimize $J(x,t,y;u,\xi)$ which is defined in \eqref{J}. Assume the following conditions hold:

    (1) $\theta \alpha_1 c_1(t) > 0,\ \ \  \forall t 
    \in [0,T]$.
    \vspace{0.5em}

    (2) $l'_u(t) \le 0,\ \  \forall t \in \{ s \in [0,T]:l_u(s) >0 \}.$
    \vspace{0.5em}

     (3) $\begin{aligned}[t]
        &  \left\{\frac{\dd [\theta\ee^{\rho(T-t)}c_0(t)]}{\dd t} +\frac{l_u(t)}{2}\frac{\dd [\alpha_1 c_1(t) \theta \ee^{\rho(T-t)}]}{\dd t} -b\ee^{-b(T-t)}\right\} \mathbf{1}_{\{l_u(t) \ge 0\}}\ge 0,\nonumber
    \end{aligned}$ 
    \vspace{0.5em}
    
\noindent where
    \begin{equation} \label{l_u}
        l_u(t) = \min\big\{\frac{\ee^{-b(T-t)} + \frac{\alpha_2 u}{b}(1-\ee^{-b(T-t)}) - \theta c_0(t) \ee^{\rho(T-t)}}{\theta \alpha_1 \ee^{\rho(T-t)}c_1(t)},\ \ \  m\big \}.
    \end{equation} \noindent
Define 
    \begin{align*}
        W^u := \{(x,t,y)\in \Qc:y > l_u(t)\},\  \quad P^u:=\{(x,t,y)\in\Qc: y \le l_u(t)\}. 
    \end{align*}
\noindent
Then the singular control law $\Xi^u$ is an equilibrium where $\Xi^u = (W^u,P^u)$, and the corresponding value function $V^u$ and the auxiliary function $g^u$ are as follows:
    \begin{equation*}
      \!\! \!\!\!\!\!\!\!\!\!\!\!\!  V^u(x,t,y)  = 
        \begin{cases}
             \ee^{-b(T-t)}x+\frac{\alpha_2 u}{b}(\ee^{-b(T-t)}-1)y\\
             + \frac{[a_L(1-u)+(a_H+\alpha_2 m)u]}{b}(1-\ee^{-b(T-t)}) +\frac{\gamma\sigma_H^2 u^2}{2b}(1-\ee^{-2b(T-t)}),\, &\forall (x,t,y) \in \overline{W^u},\\
            V^u(x-(l_u(t)-y),t,l_u(t))\\
            +\biggl[\theta (c_0(t)+\alpha_1 c_1(t)y)(l_u(t)-y) +\frac{\theta \alpha_1 c_1(t)}{2}(l_u(t)-y)^2\biggl] \ee^{\rho(T-t)},\, & \forall (x,t,y) \in P^u.
        \end{cases}
        \end{equation*}
     \!\! \!\!\!\!\!\!\!\!\!\!\!\!  \begin{equation*}
      \!\! \!\!\!\!\!\!\!\!\!\!\!\!   g^u(x,t,y)  = 
        \begin{cases}
            \ee^{-b(T-t)}x + \frac{\alpha_2 u}{b}(\ee^{-b(T-t)}-1)y + \frac{a_L(1-u)+(a_H+\alpha_2 m)u}{b}(1-\ee^{-b(T-t)}), &\forall (x,t,y) \in \overline{W^u},\\
            g^u(x-(l_u(t)-y),t,l_u(t)), & \forall (x,t,y) \in P^u.
        \end{cases}
    \end{equation*}
\end{theorem} \noindent
\begin{proof}
    The proof follows the similar arguments to that  of Theorem~\ref{example 1} by setting the admissible set of  regular control to $\{u\}$ and thus is omitted for brevity.
\end{proof}

\begin{remark}
     If $\theta \ee^{\rho(T-t)}c_0(t) = \ee^{-b(T-t)}$,  $\theta\alpha_1 \ee^{\rho(T-t)}c_1(t) = 1$, $\alpha_2$ is sufficiently small and $T$ such that $-\frac{1}{b}\ln \frac{a_L-a_H-\alpha_2 m}{2\gamma\sigma_H^2}<T<\frac{\gamma\sigma_H^2}{\alpha_2^2}$, then the conditions in Assumption~\ref{assump-solution}, Theorem~\ref{singular_only} and Theorem~\ref{example 1} are satisfied.  
\end{remark}

\begin{remark}
Unlike the mixed regular-singular control problem in Theorem~\ref{example 1}, the equilibrium singular control problem with a fixed regular control yields a single-regime free boundary without state-dependent switching. Although both scenarios share the same underlying trade-off between adjustment costs and risk exposure, the piecewise boundary structure disappears when the regular control is fixed. Without the feedback loop from the regular control, the state-dependent transitions across the free boundary are entirely eliminated.
\end{remark}
\begin{theorem}\label{regular-only}
    Under Assumption~\ref{assump-solution}, let $y \in [0,m]$ be given and fixed the singular control by $\xi_t \equiv y, t \in [0,T]$, consider the reduced regular-control-only equilibrium problem, that is, choosing a regular control $\pi$ to minimize $$J(x,t;\pi,y):=\mathbb{E}_{x,t,y}[X^{\pi,y}_{T}] + \gamma \operatorname{Var}_{x,t,y}[X^{\pi,y}_{T}].$$ Then the  equilibrium regular control is 
    \begin{equation}\label{pi^y}
    \pi^y(t) = \mathbf{1}_{\{t \le f(y)\}} + \frac{[a_L-a_H+\alpha_2 (y-m)]}{2\gamma \sigma_H^2}\ee^{b(T-t)}\mathbf{1}_{\{t>f(y)\}},
    \end{equation}
    where $f(y)$ is defined in \eqref{f(y)}, and the corresponding value function $V^y$ and the auxiliary function $g^y$ are as follows:
    \begin{equation} \label{Vy gy}
    \begin{aligned}
        V^y(x,t) =& \ee^{-b(T-t)}x + \left\{\frac{a_H-\alpha_2(y-m)}{b}(\ee^{-b(T-f(y)}-\ee^{-b(T-t)})+\frac{a_L (1-\ee^{-b(T-f(y))})}{b}\right.\\
        & \left.+\frac{\gamma\sigma_H^2}{2b}(\ee^{-2b(T-f(y)}-\ee^{-2b(T-t)})-\frac{[a_L-a_H+\alpha_2(y-m)]^2}{4\gamma\sigma_H^2}(T-f(y))\right\}\mathbf{1}_{\{t \le f(y)\}}\\
        & +\left\{\frac{a_L}{b}(1-\ee^{-b(T-t)})-\frac{[a_L-a_H+\alpha_2(y-m)]^2}{4\gamma\sigma_H^2}(T-t)\right\}\mathbf{1}_{\{t>f(y)\}},\\
       g^y(x,t)  = & \left\{ \frac{[a_H-\alpha_2(y-m)](\ee^{-b(T-f(y)}-\ee^{-b(T-t)})+a_L (1-\ee^{-b(T-f(y))})}{b}-\frac{[a_L-a_H+\alpha_2(y-m)]^2}{2\gamma\sigma_H^2}\right.\\
       &\times(T-f(y))\Biggl\}\mathbf{1}_{\{t \le f(y)\}}
        +\left\{\frac{a_L}{b}(1-\ee^{-b(T-t)})-\frac{[a_L-a_H+\alpha_2(y-m)]^2}{2\gamma\sigma_H^2}(T-t)\right\}\mathbf{1}_{\{t>f(y)\}}+ \ee^{-b(T-t)}x.\\
        \end{aligned}
    \end{equation}
\end{theorem}
\begin{proof}
    Given $y\in[0,m]$, the controlled risk exposure process follows:
    $$
    \begin{cases}
        \dd X^{\pi,y}_s = [a_L(1-\pi_s)+(a_H-\alpha_2(y-m))\pi_s-b X_s^{\pi,y}]\dd s + \sigma_H \pi_s \dd B_s,\forall s\in(t,T],\\
        X_{t-}=x,
    \end{cases}$$
    we have
    \begin{align*}
        X^{\pi,y}_T  & =  \int_t^T[a_L+(a_H-\alpha_2(y-m)-a_L)\pi_s]\ee^{-b(T-s)}\dd s + \int_t^T \sigma_H \pi_s \ee^{-b(T-s)}\dd B_s+\ee^{-b(T-t)}x,\\ J(x,t;\pi,y) & =  \ee^{-b(T-t)}x + \int_t^T[(a_L+(a_H-\alpha_2(y-m)-a_L)\pi_s)\ee^{-b(T-s)}
    + \gamma \sigma_H^2 \ee^{-2b(T-s)}\pi_s^2 ]\dd s,\\  g(x,t;\pi,y) &= \ee^{-b(T-t)}x + \int_t^T[a_L+(a_H-\alpha_2(y-m)-a_L)\pi_s]\ee^{-b(T-s)}
    \dd s.
    \end{align*}
    Then
    \begin{align*}
    \pi^y_t& = \argmin_{\pi_t\in[0,1]}\left[[a_L(1-\pi_t)+(a_H-\alpha_2(y-m)\pi_t]\ee^{-b(T-t)} + \gamma \sigma_H^2 \ee^{-2b(T-t)}\pi_t^2 \right]\\
    &= \mathbf{1}_{\{t \le f(y)\}} + \frac{[a_L-a_H+\alpha_2 (y-m)]}{2\gamma \sigma_H^2}\ee^{b(T-t)}\mathbf{1}_{\{t>f(y)\}},\quad\forall t\in[0,T],
    \end{align*}
  substituting $\pi_t = \pi^y_t$ into the expression of $J(x,t;\pi,y)$ and $g(x,t;\pi,y)$, \eqref{Vy gy} follows.
\end{proof}

\begin{remark}
The equilibrium regular control policy derived in Theorem~\ref{regular-only} has the same analytical form as the regular control component in Theorem~\ref{example 1} after replacing the accumulated reinsurance variable $\xi_{t-}$ by the prescribed level $y$. In this setting, the singular control is exogenous, so that $\xi$ enters only as a parameter rather than through a feedback equilibrium mechanism.
\end{remark}

We next examine the limiting case in which the coupling parameter $\alpha_2=0$, which leads to a complete decoupling of the equilibrium  regular-singular system.

\begin{theorem}[Decoupling of the equilibrium controls]
Under Assumption~\ref{assump-solution}, and assume that  the following conditions hold:

(1) $\alpha_2 = 0,\ \ \  0 < \theta \alpha_1 c_1(t) \ee^{\rho(T-t)}$.
\vspace{0.5cm}

(2) $y_0'(t) \le 0,\ \ \  \forall t \in\{s\in[0,T]: y_0(t) > 0\}$.
\vspace{0.5cm}

(3) $\begin{aligned}[t]
        &  \left\{\frac{\dd [\theta\ee^{\rho(T-t)}c_0(t)]}{\dd t} +\frac{y_0(t)}{2}\frac{\dd [\alpha_1 c_1(t) \theta \ee^{\rho(T-t)}]}{\dd t} -b\ee^{-b(T-t)}\right\} \mathbf{1}_{\{ y_0(t) \ge 0\}}\ge 0,\nonumber
    \end{aligned}$
    \vspace{0.5cm}
    
\noindent where
\begin{equation}\label{y0}
    \begin{aligned}
        y_0(t) := \frac{\ee^{-b(T-t)}-\theta \ee^{\rho(T-t)}c_0(t)}{\theta \ee^{\rho(T-t)}c_1(t)}.
    \end{aligned}
\end{equation}
Define \begin{equation*}
    \begin{aligned}
       & W_0:=\{(x,t,y)\in\Qc:y > y_0(t)\},\,\ \  P_0 := \{(x,t,y)\in\Qc : y\le y_0(t) \},\\
       & \Pi_0(x,t,y) :=  \frac{a_L-a_H}{2\gamma\sigma_H^2}\ee^{b(T-t)} \mathbf{1}_{\{t>\tau\}} + \mathbf{1}_{\{t \le \tau\}}, \label{Pi0}
    \end{aligned}
\end{equation*}
where $\tau:=T+\frac{1}{b}\ln{\frac{a_L-a_H}{2\gamma\sigma_H^2}}$. Then the regular-singular control law $(\Pi_0,\Xi_0)$ is an equilibrium where $\Xi_0 = (W_0, P_0)$, and the corresponding value function $V_0$ and the auxiliary function $g_0$ are as follows:
\begin{equation*}
\begin{aligned}
  V_0(x,t,y) = 
    \begin{cases}
          \left\{ \frac{a_H}{b}(\ee^{-b(T-\tau)}-\ee^{-b(T-t)})+\frac{a_L (1-\ee^{-b(T-\tau)})}{b}+\frac{\gamma\sigma_H^2}{2b}(\ee^{-2b(T-\tau)}-\ee^{-2b(T-t)})-\frac{(a_L-a_H)^2}{4\gamma\sigma_H^2}(T-\tau)\right\}\mathbf{1}_{\{t \le \tau\}}\\
       \hspace{1cm} +\left\{\frac{a_L(1-\ee^{-b(T-t)})}{b} + \frac{(a_H-a_L)^2 (t-T)}{4\gamma\sigma_H^2}\right\}\mathbf{1}_{\{t > \tau\}} + \ee^{-b(T-t)}x,\quad \forall (x,t,y) \in \overline{W_0},\\
         V_0(x-(y_0(t)-y),t,y_0(t))+\biggl[\theta (c_0(t)+\alpha_1 c_1(t)y)(y_0(t)-y) \\
         \hspace{5cm}+\frac{\theta \alpha_1 c_1(t)}{2}(y_0(t)-y)^2]\biggl] \ee^{\rho(T-t)}, \quad \forall (x,t,y) \in P_0, 
    \end{cases}
    \end{aligned}
\end{equation*}
\begin{equation*}
\begin{aligned}
 g_0(x,t,y)  =
    \begin{cases}
        \ee^{-b(T-t)}x + \left\{   \frac{a_H(\ee^{-b(T-\tau)}-\ee^{-b(T-t)})}{b}+\frac{a_L (1-\ee^{-b(T-\tau)})}{b}-\frac{(a_L-a_H)^2(T-\tau)}{2\gamma\sigma_H^2}\right\}\mathbf{1}_{\{t \le \tau\}}\\
        \   +\left\{\frac{a_L}{b}(1-\ee^{-b(T-t)})-\frac{(a_L-a_H)^2}{2\gamma\sigma_H^2}(T-t)\right\}\mathbf{1}_{\{t>\tau\}}, \quad \forall (x,t,y)\in \overline{W_0},\\
        g_0(x-(y_0(t)-y),t,y_0(t)), \hspace{3.95cm} \forall (x,t,y) \in P_0.
    \end{cases}
\end{aligned}
\end{equation*}
\end{theorem}
\begin{proof}
    The proof follows the same verification framework as in Theorem~\ref{example 1} and is thus omitted for brevity.
\end{proof}

\begin{remark}
    It can be observed from the explicit expressions of $y^*(t)$ and $y^{*1}(t)$ (see \eqref{y^*},\eqref{y^*1}) that 
    $$\lim_{\alpha_2 \rightarrow 0+} y^*(t;\alpha_2) = \lim_{\alpha_2 \rightarrow 0+} y^{*1}(t;\alpha_2) = y_0(t).$$
    This indicates a structural consistency between the coupled and decoupled regimes at the level of free boundaries. In particular, when $\alpha_2 =0 $, the  switching free boundary degenerates as $y^*(t) = y^{*1}(t) = y_0(t)$, that is, the regular-control switching structure decouples from the free boundary determination and therefore does not affect its characterization.
\end{remark}

\begin{proposition}(Equivalence between joint and separate optimization).
\label{decoupled}
Suppose that $\alpha_2=0$. Then the equilibrium controls obtained from the mixed regular-singular control problem coincide with those obtained from the corresponding single-control problems. More precisely,
$$
\pi_0=\pi^{y}, \,\,\forall y \in [0,m],
\qquad
y_0 = l_u, \,\,\forall u \in [0,1],
$$
where $y_0, l_u, \pi^y$ are defined in    \eqref{y0}, \eqref{l_u} and  \eqref{pi^y}, respectively, and $\pi_0(x,t,y) = \Pi_0(x,t,y)$ which is defined in \eqref{Pi0}. Consequently, the equilibrium strategy of the mixed control problem can be constructed by combining the equilibrium regular control and equilibrium singular control obtained separately.
\end{proposition}

\begin{remark}
Proposition~\ref{decoupled} reveals the structural role of the coupling parameter $\alpha_2$. In the coupled model ($\alpha_2>0$), the term $\alpha_2(y-m)$ enters the drift of the state process and establishes a transmission channel from the singular control to the regular control. Consequently, the equilibrium regular control policy depends on the reinsurance level $\xi$, which induces a nontrivial dependence of the equilibrium value function on $\xi$. This dependence is further reflected in the gradient condition characterizing the singular control, thereby affecting the free boundary.

In the absence of this parameter, the equilibrium regular control strategy does not respond to the singular control state. Consequently, the regular control has no impact on the functional form of the value function with respect to $y$ and the free boundary is identical to that obtained from the corresponding pure singular control problem. 
\end{remark}

The above degenerate analysis is carried out under a specific reinsurance model, where the equilibrium value functions admit a special linear structure in $x$. In a more general setting, such degenerate decoupling results may not hold if the equilibrium value function is nonlinear in $x$. 

\paragraph{Structural Insight.}
The explicit solutions reveal that the equilibrium regular and singular controls cannot be characterized independently. Their interaction generates an endogenous state-space partition jointly determined by the regular-control switching boundary and the singular-control free boundary. Moreover, the decoupling results show that this regime-dependent structure disappears when the coupling parameter vanishes. Therefore, the observed equilibrium geometry is an intrinsic consequence of the interaction between regular and singular controls.


\printbibliography

\appendix
\label{appendix}
\section{Proof of Proposition \ref{square int}}
\label{pro2.1}
\begin{proof}
To establish the square integrability and preclude finite-time explosion, we introduce a sequence of stopping times defined by $\tau_n = \inf\{s\in[t,T]:|X^{\pi,\xi}_s| \ge n\}$ with the convention $\inf \emptyset = T$. Then the stopped process of  $X^{\pi,\xi}$ satisfies the following integral equation for $s \in [t, T]$:
    $$X^{\pi,\xi}_{s\wedge \tau_n} = x + \int_t^{s\wedge \tau_n} \mu(X^{\pi,\xi}_{r},r,\xi_{r},\pi_{r}) \mathrm{d} r + \int_t^{s\wedge \tau_n} \sigma(X^{\pi,\xi}_{r},r,\xi_{r},\pi_{r}) \mathrm{d} B_r - (\xi_{s\wedge \tau_n}-y).$$
By applying the Cauchy-Schwarz inequality and invoking the  growth condition \eqref{Lipschitz}, the drift term can be estimated as follows:
    \begin{equation*}
        \mathbb{E}_{x,t,y}\left[\sup_{s\in[t,T]}\left(\int_t^{s\wedge \tau_n} \mu(X^{\pi,\xi}_r,r,\xi_r,\pi_r) \mathrm{d} r\right)^2\right]  \le KT\left(T+ \int_t^{T} \mathbb{E}_{x,t,y}\left[\sup_{\tau \in [t,r\wedge \tau_n]} (X_{\tau}^{\pi,\xi})^2\right]\mathrm{d} r \right).
    \end{equation*}
    Meanwhile, utilizing Doob's martingale inequality for the stochastic integral term yields
    \begin{align*}
        \mathbb{E}_{x,t,y}\left[\sup_{s\in[t,T]}\left(\int_t^{s\wedge \tau_n} \sigma(X^{\pi,\xi}_r,r,\xi_r,\pi_r) \mathrm{d} B_r\right)^2\right] &\le 4 \mathbb{E}_{x,t,y}\left[\int_t^{T\wedge \tau_n} \sigma^2(X^{\pi,\xi}_r,r,\xi_r,\pi_r) \mathrm{d} r\right] \\
        &\le 4K\mathbb{E}_{x,t,y}\left[\int_t^T \left(1+\sup_{\tau \in [t,r\wedge \tau_n]}(X^{\pi,\xi}_\tau)^2\right) \mathrm{d} r\right].
    \end{align*}
        Applying the basic algebraic inequality $(a+b+c+d)^2 \le 4(a^2+b^2+c^2+d^2)$ and using the boundedness of the singular control process, i.e., $|\xi_{s\wedge \tau_n}-y|^2 \le b^2$, we obtain the following uniform estimate for the stopped state process:
    \begin{align*}
        \mathbb{E}_{x,t,y}\left[\sup_{s\in[t,T]} |X^{\pi,\xi}_{s \wedge \tau_n}|^2 \right] &\le 4\left(x^2+ b^2  + \mathbb{E}_{x,t,y}\left[\sup_{s\in[t,T]}\left(\int_t^{s\wedge \tau_n} \mu \mathrm{d} r\right)^2\right] + \mathbb{E}_{x,t,y}\left[\sup_{s\in[t,T]}\left(\int_t^{s\wedge \tau_n} \sigma \mathrm{d} B_r\right)^2\right]\right)\\
           & \le M + 4K(4+T)\int_t^{T} \mathbb{E}_{x,t,y}\left[\sup_{\tau \in [t,r\wedge \tau_n]} (X_{\tau}^{\pi,\xi})^2\right]\mathrm{d} r.
    \end{align*}
    where $M := 4(x^2 + b^2 + KT^2 + 4KT)$. Using Gronwall inequality yields
    \begin{equation*}
        \mathbb{E}_{x,t,y}\left[\sup_{s\in[t,T]} |X^{\pi,\xi}_{s \wedge \tau_n}|^2 \right] \le M \mathrm{e}^{4KT(4+T)}.
    \end{equation*}
    Using Markov's inequality, we have
    $$\mathbb{P}_{x,t,y}(\tau_n < T) = \mathbb{P}_{x,t,y}\left(\sup_{s\in[t,T]}|X^{\pi,\xi}_{s \wedge \tau_n}|^2 \ge n^2\right) \le \frac{\mathbb{E}_{x,t,y}\left[\sup_{s\in[t,T]}|X^{\pi,\xi}_{s \wedge \tau_n}|^2\right]}{n^2} \le \frac{M \mathrm{e}^{4KT(4+T)}}{n^2} \xrightarrow{n \rightarrow +\infty} 0.$$
     Finally,  using Fatou's lemma and the last inequality,  we obtain
     $$\mathbb{E}_{x,t,y}\left[\sup_{s\in[t,T]} |X^{\pi,\xi}_{s }|^2 \right] = \mathbb{E}_{x,t,y}\left[\lim_{n \rightarrow +\infty}\sup_{s\in[t,T]} |X^{\pi,\xi}_{s \wedge \tau_n}|^2 \right] \le \liminf_{n \rightarrow+\infty}\mathbb{E}_{x,t,y}\left[\sup_{s\in[t,T]} |X^{\pi,\xi}_{s \wedge \tau_n}|^2 \right] \le M \mathrm{e}^{4KT(4+T)}<+\infty.$$
    The integrability of the cost function $c$ follows similarly from the  growth bound \eqref{Lipschitz}, which completes the proof.
\end{proof}

\section{Equilibrium Solution With Full Terminal Jump}\label{B}
    We come to the condition of $V^{**}(x,t,y) = x-(m-y))\left\{1-\frac{\theta \alpha_1 c_1(T) (m-y)}{2}-\theta[c_0(T) + \alpha_1 c_1(T)y]\right\}$ when $\theta \alpha_1 c_1(T)m \le 1-\theta c_0(T)$, that is, at the terminal time, the insurer will optimally purchase all remaining reinsurance capacity.
\begin{theorem}[Explicit equilibrium solution under full terminal jump] \label{example 2}
    Assume that the following conditions hold:

        (1) $\alpha_2 \neq 0, \,\,\sigma_L=0,\,\, a_L-a_H\ge \alpha_2 m$.
    \vspace{0.5em}
    
    (2) $  \theta \alpha_1 c_1(T) m \le 1-\theta c_0(T),\,\, c_0(t)>0, \,\, \theta \alpha_1 c_1(t) \ge 0, \,\,\forall t\in[0,T]$.
    \vspace{0.5em}

     (3) $
    \begin{aligned}[t]
        & \theta \alpha_1 c_1(t) \ee^{\rho (T-t)}- \theta \alpha_1 c_1(T) - \frac{\alpha_2^2 (T-t)}{2\gamma \sigma_H^2} >0,\,\, \forall t \in [0,T),\\
        & \theta\alpha_1 c_1(t) \ee^{\rho (T-t)}- \theta\alpha_1 c_1(T) + \frac{\alpha_2^2}{2b\gamma\sigma_H^2}\ln{\frac{(a_L-a_H-\alpha_2 m)}{2\gamma\sigma_H^2}}\ge 0,\,\,\,\forall t \in[0,T].
    \end{aligned}$
   \vspace{0.5em}

    (4) $\tilde{k}'(t) \le 0,\, \,\forall t \in \{s\in [0,T):\tilde{k}(t)>0\}$.
    \vspace{0.5em}

    (5) $\begin{aligned}[t]
        &  \left\{\frac{\dd [\theta\ee^{\rho(T-t)}c_0(t)]}{\dd t} +\frac{\tilde{k}(t)}{2}\frac{\dd [\alpha_1 c_1(t) \theta \ee^{\rho(T-t)}]}{\dd t} -b\ee^{-b(T-t)}\right\} \mathbf{1}_{\{\tilde{k}(t) \ge 0\}}\ge 0,\nonumber\\
        &y^{**}(0) \ge \hat{y}(0),
    \end{aligned}$ 
    
    where 
    \begin{equation*}
        \begin{aligned}
            y^{**}(t) & := \min\left\{\frac{\theta c_0(t) \ee^{\rho(T-t)} - \ee^{-b(T-t)} - \frac{\alpha_2(t-T)}{2\gamma \sigma_H^2}(a_H-a_L+\alpha_2 m)+[1-\theta c_0(T)]}{\frac{\alpha_2^2(T-t)}{2\gamma\sigma_H^2}-\theta \alpha_1 c_1(t)\ee^{\rho(T-t)}+\theta \alpha_1 c_1(T)}, \ \ m \right\}, 
            \end{aligned}
            \end{equation*}

            \begin{equation*}
        \begin{aligned}
            y^{**1}(t) & := \inf\Biggl\{y\in[0,m]:\ \ \frac{a_L-a_H+\alpha_2(y-m)}{2\gamma\sigma_H^2}\frac{\alpha_2}{b}\left[-1+ \ln \left(\frac{a_L-a_H+\alpha_2(y-m)}{2\gamma\sigma_H^2}\right)\right]\\
            &\quad\quad\quad\quad +\left(\frac{\alpha_2}{b}-1\right)\ee^{-b(T-t)}+\theta \ee^{\rho(T-t)}[c_0(t) +\alpha_1 c_1(t)y] + 1-\theta c_0(T) - \theta \alpha_1 c_1(T)y > 0\Biggl \},\\
            \tilde{k}(t) & := 
            \begin{cases}
        y^{**}(t),\quad &\forall t\in\{t\in[0,T):y \le \hat{y}(t)\},\\
        y^{**1}(t),\quad &\forall t\in \{t\in[0,T):y > \hat{y}(t)\},\\
        0,& t= T.
        \end{cases}
        \end{aligned}
    \end{equation*}
    Define 
    \begin{equation*}
        \begin{aligned}
           & W^{**}:= \{(x,t,y) \in \Qc: y > \tilde{k}(t), t <T \text{ or } y = m\},\, P^{**}:=\{(x,t,y)\in \Qc:y \le \tilde{k}(t), y < m \text{ or } t=T, y < m\},\\
           & \Pi^{**}(x,t,y):=\min \left\{ \frac{[a_L-a_H+\alpha_2 (y-m)]}{2\gamma \sigma_H^2}\ee^{b(T-t)}, \, 1\right\} = \mathbf{1}_{\{y\ge\hat{y}(t)\}} + \frac{[a_L-a_H+\alpha_2 (y-m)]}{2\gamma \sigma_H^2}\ee^{b(T-t)}\mathbf{1}_{\{y < \hat{y}(t)\}}, 
        \end{aligned}
    \end{equation*}
Then the regular-singular control law $(\Pi^{**},\Xi^{**})$ is an equilibrium where $\Xi^{**} = (W^{**},P^{**})$, and the corresponding value function $V^{**}$ and the auxiliary function $g^{**}$ are as follows:
        \begin{equation} \label{V**}
       V^{**}(x,t,y)  = 
        \begin{cases}
             \ee^{-b(T-t)}x+\frac{(a_H-a_L-\alpha_2(y-m))^2}{4\gamma \sigma_H^2}(t-T) + \frac{a_L}{b}(1-\ee^{-b(T-t)})\\
             \hspace{2cm}-(m-y)\left\{1-\frac{\theta \alpha_1 c_1(T) (m-y)}{2}-\theta[c_0(T) + \alpha_1 c_1(T)y]\right\},\,  \forall (x,t,y) \in \overline{W^{**}\cap \Bc},\\
             (\ee^{-b(T-t)} - \ee^{-b(T-f(y)})x + (\ee^{-b(T-f(y))}-\ee^{-b(T-t)})\frac{a_H+\alpha_2(m-y)}{b}\\
             \hspace{3cm}+(\ee^{-2b(T-f(y))}-\ee^{-2b(T-t)})\frac{\gamma \sigma_H^2}{2b}+V^{**}(x,f(y),y),\,\forall (x,t,y) \in \overline{W^{**}}\cap \Bc^c, \\
            V^{**}(x-(k(t)-y),t,k(t))+\biggl[\theta (c_0(t)+\alpha_1 c_1(t)y)(k(t)-y)\\
           \hspace{7cm} +\frac{\theta \alpha_1 c_1(t)}{2}(k(t)-y)^2]\biggl] \ee^{\rho(T-t)},\, \forall (x,t,y) \in P,
        \end{cases}\\
        \end{equation}
        \begin{equation}\label{g**}
        g^{**}(x,t,y)  = 
        \begin{cases}
            \ee^{-b(T-t)}x + \frac{(a_L-a_H+\alpha_2(y-m))^2}{2\gamma \sigma_H^2}(t-T) +\frac{a_L}{b}(1-\ee^{-b(T-t)}), & \forall (x,t,y) \in \overline{W^{**}\cap \Bc},\\
            (\ee^{-b(T-t)}-\ee^{-b(T-f(y))}x + (\ee^{-b(T-f(y))}-\ee^{-b(T-t)}) \frac{a_H+\alpha_2(m-y)}{b} \\
            \hspace{7.5cm}+g^{**}(x,f(y),y),  & \forall (x,t,y) \in \overline{W^{**}}\cap \Bc^c\\
            g^{**}(x-(k(t)-y),t,k(t)), & \forall  (x,t,y) \in P^{**},
        \end{cases}
    \end{equation}
    where $\Bc = \{(x,t,y)\in\Qc:t > f(y)\}$, $\Bc^c$ is the complement of $\Bc$ in $\Qc$, and $f(y)$ is defined in \eqref{f(y)}.
\end{theorem}

\begin{proof}
    The proof follows the same arguments as those of Theorem~\ref{example 1} and is thus omitted for brevity.
\end{proof}

\begin{remark}
The parameter specification considered in Remark~\ref{c1 c_2 ex1} does not satisfy Condition (1) of Theorem~\ref{example 2}. To illustrate the applicability of the present theorem, consider instead
$$ \theta c_0(t) \ee^{\rho(T-t)} = \ee^{-b(T-t)},\, \ \ \ \ \theta \alpha_1 c_1(t) \ee^{\rho(T-t)} = \frac{\alpha_2^2(T-t)}{\gamma\sigma_H^2}.$$
If $\alpha_2 m \le \min\{2\gamma\sigma_H^2\ee^{-bT},\,\, a_L-a_H-2\gamma\sigma_H^2\}$, it is readily verified that the  conditions in Theorem~\ref{example 2} are satisfied, and hence the explicit solution characterized in Theorem~\ref{example 2} arises under this parameter regime.
\end{remark}
\end{document}